\documentclass[pdflatex,sn-mathphys-num]{sn-jnl}

\usepackage{graphicx}%
\usepackage{multirow}%
\usepackage{amsmath,amssymb,amsfonts}%
\usepackage{amsthm}%
\usepackage{mathrsfs}%
\usepackage[title]{appendix}%
\usepackage{xcolor}%
\usepackage{textcomp}%
\usepackage{manyfoot}%
\usepackage{booktabs}%
\usepackage{algorithm}%
\usepackage{algorithmicx}%
\usepackage{algpseudocode}%
\usepackage{listings}%
\usepackage{subcaption}

\hypersetup{
  colorlinks=true,
  linkcolor=blue,
  citecolor=blue,
  urlcolor=blue
}

\numberwithin{equation}{section}

\newtheorem{theorem}{Theorem}

\newtheorem{assumption}{Assumption}
\newtheorem{lemma}{Lemma}
\newtheorem{example}{Example}
\newtheorem{Remark}{Remark}

\let\orcidlogo\relax
\usepackage{orcidlink}
\begin{document}

\title[Multiscaling asymptotic of random high-order heat equations]{Multiscaling asymptotic behavior of solutions to random high-order heat equations}


\author[1,2]{\fnm{Maha Mosaad A} \sur{Alghamdi}\orcidlink{0009-0004-6663-2286}}\email{20312612@students.latrobe.edu.au; mmghamdi@iau.edu.sa}\equalcont{These authors contributed equally to this work.}

\author[3]{\fnm{Nikolai} \sur{Leonenko}\orcidlink{0000-0003-1932-4091}}\email{LeonenkoN@cardiff.ac.uk}
\equalcont{These authors contributed equally to this work.}

\author*[1]{\fnm{Andriy} \sur{Olenko}\orcidlink{0000-0002-0917-7000}}\email{a.olenko@latrobe.edu.au}\equalcont{These authors contributed equally to this work.}

\affil[1]{\orgdiv{Department of Mathematical and Physical Sciences}, \orgname{La Trobe University}, \orgaddress{\city{Melbourne}, \postcode{3086}, \state{VIC}, \country{Australia}}}

\affil[2]{\orgdiv{Department of Mathematics, College of Science and Humanities}, \orgname{Imam Abdulrahman Bin Faisal University}, \orgaddress{\city{Jubail}, \postcode{31441}, \country{Saudi Arabia}}}

\affil[3]{\orgdiv{School of Mathematics}, \orgname{Cardiff University}, \orgaddress{\street{Senghennydd Road}, \city{Cardiff}, \postcode{CF24 4YH}, \country{United Kingdom}}}


\abstract{This paper studies high-order partial differential equations with random initial conditions that have both long-memory and cyclic behavior. The cases of random initial conditions with the spectral singularities, both at zero (representing classical long-range dependence) and at non-zero frequencies (representing cyclic long-range dependence), are investigated.
  Using spectral methods and scaling techniques, it is proved that, after proper rescaling and normalization, the solutions converge to Gaussian random fields. For each type of equation, spectral representations and covariance functions of limit fields are given. For odd-order equations, we apply the kernel averaging of solutions to obtain nonexplosive and nondegenerate limits. It is shown that the different limit fields are determined by the even or odd orders of the equations and by the presence or absence of a spectral singularity at zero. Several numeric examples illustrate the obtained theoretical results.}

\keywords{High order heat equation, Airy equation,  Random partial differential equations, Spectral singularities, Multiscaling, Limit theorems}



\maketitle

\section{Introduction}\label{sec1}

High-order generalisations of the heat equation have found numerous physical applications, in particular, in modelling wave propagation in fluids, plasma, and other media; anomalous diffusion; higher-order dispersion; and systems where standard diffusion is insufficient. Higher-order terms allow for an adequate description of these physical systems, see \cite{Peletier, Gorska, kozachenko2018estimates}, the references therein, and the specific references provided below.

Partial differential equations with random initial conditions have been the focus of extensive research in biosciences, engineering, physics, and mathematics. Early studies by De Fériet \cite{de1956random} and Rosenblatt \cite{rosenblatt1968remarks} introduced rigorous probabilistic techniques to analyze the heat equation when its initial conditions were given as a stationary random process. These results established representations of the solutions as stochastic integrals. Later contributions extended this framework to cases with more general types of random initial conditions, including non-homogeneous and random potential scenarios, see \cite{becus1980variational, uboe1995stability}.

Further research has considered diffusion equations driven by random processes and fields with a singular spectrum at the origin. This includes publications on scaling of linear equations with singular initial conditions, see \cite{Anh1999NonGaussianSF, Leonenko1998ScalingLO}. The scaling laws for the linear KdV and Airy equations with weakly dependent Gaussian random inputs were studied in \cite{beghin2000}. Related methods were applied to the renormalization and homogenization of fractional diffusion equations, as discussed in~\cite{ANH2000239, Anh2002RenormalizationAH}. See also \cite{10.1214/EJP.v16-896} and the references therein. These results generalized the classical limit theorems of Dobrushin, Major \cite{dobrushin1979non} and Taqqu~\cite{taqqu1979convergence}.

Similar problems were also studied for stochastic differential equations defined on spheres. Several classes of hyperbolic and fractional diffusion equations were investigated, motivated by cosmological applications to cosmic microwave background radiation. The publications~\cite{broadbridge2019random, BKLO1, anh2021fractional} analyzed a hyperbolic diffusion model on the unit sphere, deriving exact series solutions of the Cauchy problem with random initial conditions. Another recent publication \cite{Broadbridge2024} examined the Cauchy problem for random diffusion within an expanding space–time framework. In addition to analyzing various probabilistic properties of the solutions, it characterized the extremal behavior of the solutions by establishing upper bounds on the probabilities of large deviations.

Kozachenko et al. \cite{kozachenko2018estimates} investigated high-order heat-type equations for the case of  $\varphi$-sub-Gaussian random initial conditions. They derived upper bounds for the distributions of the suprema of solutions. Then, in the paper~\cite{kozachenko_orsingher_sakhno_vasylyk_2019}, they generalized these results to high-order dispersive partial differential equations and obtained explicit estimates of the suprema and their convergence conditions.

The concept of scaling limits has been extensively applied in statistical physics to describe macroscopic behavior, see, for example, \cite{Leonenko1998ScalingLO, Giuliani2023} and references therein. The publication \cite{knopova2004limit} examined the Airy equation with strongly dependent Gaussian initial conditions and established convergence of the renormalized solution via appropriate scaling. The publication~\cite{Airy} expanded this approach under general structural assumptions regarding dependence and developed a renormalization strategy to find the limiting behavior of the solution.

The paper \cite{cinque2025general} studied non-random high-order Airy-type and heat-type equations, developing explicit solutions through power series and integral representations. It established a probabilistic interpretation of the analytical results by linking them to pseudo-processes. The article
\cite{marchione2024stable} investigated pseudo-processes governed by odd-order equations involving the Riesz–Feller operator. They introduced extended Airy functions via power series expansions and proved that they represent pseudo-densities of asymmetric stable processes via fractional time-changed subordinators.

All the mentioned results considered the cases of short-range dependent (no spectral singularities) or long-range dependent (a spectral singularity at the origin) initial conditions. However, various real-world time series exhibit periodic or seasonal behavior and long-range dependence. Modeling and studying such scenarios has been an area of active research in recent decades. Numerous early studies have identified simultaneous cyclical behavior and strong dependence in various data, see, for example,~ \cite{porter1990application,ray1993long}, leading to the development of several cyclic long-memory models. Arteche and Robinson \cite{arteche1999seasonal, arteche2000semiparametric} contributed significantly to the theory and estimation of these processes by developing inference techniques. Gegenbauer autoregressive moving average models and their extensions (e.g. SARFIMA and ARFISMA) are the most popular approaches that provide an effective analytical framework for analyzing data with cyclic and long-memory effects, see
\cite{chung1996generalized,gray1989generalized, alomari2020estimation, Olenko2022}.

Recently, for the classic heat equation, the aforementioned results were expanded to the cases where the initial condition is a random process that exhibits cyclic long-range dependence, see~\cite{alghamdi2024}. It is characterized by a spectral singularity at a non-zero frequency, leading to asymptotic behaviors that differ from the classical long-memory cases, see \cite{Ivanov2013, Olenko2013}.

Further important extensions include studying data with multiple cycles (several spectral singularities) and long-memory effects, which require further development of the available models and new limit theorems and statistical methods. The paper generalizes known results in two directions.
First, it considers the cases of multiple spectral singularities and proves that the limits are different depending on the presence or absence of a spectral singularity at the origin. Also, the normalizing factors are different in these cases.
If the spectrum has a singularity at the origin, it determines the asymptotic behavior dominating all other singularities.
Secondly, it is shown that the asymptotic behavior and limit theorems are different for odd- and even-order heat equations.
For the even-order case, the limit fields are stationary in space but non-stationary in time.
For the odd-order case, analogous limit results do not exist in principle. Therefore, for this case, it is proposed to study kernel–smoothed solutions. Then, the obtained limits are stationary both in space and~time.

  The article has the following structure. Section \ref{sec2} presents the main definitions, assumptions, and two classical equations that motivated this research. Section~\ref{sec heat} proves the main results for even-order heat equations under cyclic long-range dependent scenarios. Section~\ref{secodd} provides the discussion of difficulties for the case of odd-order equations. Then, it proves the corresponding limit theorems.  The paper also presents numerical examples that illustrate the obtained results. Some future research directions are discussed in Section \ref{sec conclusion}.

\section{Preliminaries}
\label{sec2}

The classical heat and Airy equations with random initial conditions have attracted the attention of researchers in statistical physics, stochastic processes, stochastic partial differential equations, and their applications. This section introduces the main notations and presents these equations and a recent result on their asymptotic behavior, which motivated generalizations in the following sections.

Let $\xi(x)=\xi(x,\omega),$ $ x\in \mathbb R,$ be a real-valued, measurable, mean-square continuous stationary Gaussian random process, defined on a complete probability space $(\Omega, \mathscr{F},\mathbb{P}),$ with $\mathbb E\xi(x)=0$ and $\mathbb E\xi^2(x)=1 $.

The absolutely continuous spectral measure $F_{\xi}(\cdot)$  can be represented as
\[F_{\xi}(\Delta)=\int_{\Delta} f_{\xi}(\lambda)d\lambda,\quad \Delta\in \mathscr{B}(\mathbb R). \]
The function $f_{\xi}(\cdot)$ is integrable over $\mathbb R$ and called the spectral
density function of the random process $\xi(\cdot)$.
Then, the spectral representation of the covariance function
 can be written as
\begin{equation}\label{cov}
    B_\xi(x) = \int_{\mathbb R}e^{i\lambda x}f_{\xi}(\lambda)d\lambda.
\end{equation}

 The random process $\xi(x)$ has the following isonormal spectral representation
\begin{equation*}\label{spec_xi}\xi(x)=\int_{\mathbb R}e^{i \lambda x } Z(d\lambda)=\int_{\mathbb R}e^{i \lambda x } \sqrt{f_{\xi}(\lambda)}W(d\lambda),
\end{equation*}
 where $Z(\cdot)$ and $W(\cdot)$ denote the random measure with the control measure~$F(\cdot)$ and the white-noise random measure on $\mathbb{R}$ respectively. The integrals are interpreted in $L_2(\Omega)$ sense. For further details, see  \cite{Statisticalanalysis}.

Let us consider the classical heat equation
\begin{equation}\label{classical heat equation}
   \frac{\partial u(t,x)}{\partial t}=\mu\frac{\partial^2 u(t,x)}{\partial x^2},\quad t>0,\quad \mu>0, \quad x\in \mathbb R,
\end{equation}
subject to the random initial condition
\begin{equation}\label{initial Airy}
u(0, x)=\eta(x) =M(x)+\xi(x)=M(x)+
\int_{\mathbb{R}} e^{i\lambda x} \, Z(d\lambda),
\end{equation}
 where $M(\cdot)$ is a deterministic function.

If $M(x)\equiv 0,$ the solution field $u(t,x)$ has a zero mean, the spectral representation
\begin{equation*}\label{convolution}
    u(t,x)=\int_{\mathbb R}e^{i\lambda x-\mu \lambda^{2}t}Z(d\lambda),
\end{equation*}
 and its covariance function equals to
\[
    {\rm Cov}(u(t,x),u(t',x'))=\int_{\mathbb R}e^{i \lambda (x-x')-\mu\lambda^{2}(t+t')}F(d\lambda).
\]
\begin{assumption}\label{AssA}  The covariance function {\rm(\ref{cov})} has the form
\begin{equation*}\label{covariance of our problem1}
    B_{\kappa,w}(x)=\frac{\cos(wx)}{(1+x^{2})^{\kappa/2}},\quad x\in\mathbb R,\quad w\neq 0,\quad 0<\kappa<1.
\end{equation*}
\end{assumption}
Under this assumption on random initial conditions, the next recent result established the multiscaling asymptotic behavior of the equation solutions.
\begin{theorem}{\label{the3}} {\rm\cite{alghamdi2024}}  Consider the random field $u(t,x),$ $t>0,$ $x\in \mathbb  R,$ that is the solutions of the initial value problem {\rm (\ref{classical heat equation})} and {\rm(\ref{initial Airy})} with $M(x)\equiv 0.$  If Assumption~\ref{AssA} holds true, then,  for $\varepsilon \to 0,$  the finite-dimensional distributions of the random fields
    \begin{equation*}
        U_{\varepsilon}(t,x)=\varepsilon^{-1/4}\,u\left(\frac{t}{\varepsilon}, \frac{x}{\sqrt{\varepsilon}}\right),
    \end{equation*}
    converge weakly to the finite-dimensional distributions of the zero-mean Gaussian random field
    \begin{equation*}\label{*1}
       U_{0}(t,x)=\sqrt{\frac{c_2(\kappa)}{w^{1-\kappa}}(1-\theta_\kappa(|w|))}\int_{\mathbb R}e^{i\lambda x-\mu t\lambda^{2}}W(d\lambda),
   \end{equation*}
 with the covariance function
 \begin{align*}\label{cov3}
     {\rm Cov}( U_{0}&(t,x), U_{0}(t',x'))=\frac{c_{2}(\kappa)\sqrt{\pi}(1-\theta_\kappa(|w|))}{\sqrt{\mu}w^{1-\kappa}}\cdot\frac{e^{-\frac{(x-x')^2}{4\mu(t+t')}}}{\sqrt{t+t'}},
 \end{align*}
 where $\theta_\kappa(|w|):=1-\frac{c_1(\kappa)}{c_2(\kappa)}K_{\frac{\kappa-1}{2}}\left(|w| \right)|w|^{\frac{1-\kappa}{2}},$  $c_1(\kappa):={2^{\frac{1-\kappa}{2}}}/\left({\sqrt{\pi}{\Gamma\left(\kappa/2 \right)}}\right)$, and $c_{2}(\kappa):=\left(2\Gamma(\kappa)\cos\left(\kappa\pi/2\right)\right)^{-1}.$
\end{theorem}

The first research direction of the paper is to extend the results of Theorem~\ref{the3} to higher-order heat equations, where the second-order derivative in (\ref{classical heat equation}) is replaced by an arbitrary even-order derivative. A generalized version of Assumption~\ref{AssA} is introduced, in which the covariance function is a sum of several terms of the same form as in Assumption~\ref{AssA}. It also includes scenarios where $\omega=0$ in one of these terms. Whereas Assumption~\ref{AssA} corresponds to a single non-zero spectral singularity, the more general case of the random conditions with multiple arbitrary spectral singularities will be addressed.

The second research direction concerns odd-order equations. In particular, it generalizes the third-order case of the classical Airy equation
    \begin{equation}\label{Airy equation}
\frac{\partial u}{\partial t} = -\frac{\partial^3 u}{\partial x^3}, \quad t > 0,\ x \in \mathbb{R},
\end{equation}
subject to the random initial condition
(\ref{initial Airy}).
The corresponding solution random field
\begin{align}\label{Random Airy}
u(t, x) &= \int_{\mathbb{R}} \eta(x - y) G(t, y) \, dy= \frac{1}{\sqrt{\pi}} \int_{\mathbb{R}} M(x - y) \frac{1}{\sqrt[3]{3t}} \operatorname{Ai_{3}}\left( \frac{y}{\sqrt[3]{3t}} \right) dy \nonumber \\
& \quad + \int_{\mathbb{R}} \exp\{i \lambda x + i \lambda^3 t\} dZ(\lambda),
\end{align}
is known as the Airy random field \cite{beghin2000},
where $G(t,x)$ is the fundamental solution to the Airy equation, given~by
\begin{equation}\label{Green Airy}
G(t, x) = \frac{1}{2\pi} \int_{\mathbb{R}} e^{-i\alpha x - i\alpha^3 t} \, d\alpha = \frac{1}{\sqrt[3]{3t}}\,\mathrm{Ai_{3}}\left(\frac{x}{\sqrt[3]{3t}}\right), \quad t > 0,
\end{equation}
and $\mathrm{Ai_{3}}(x)$ denotes the Airy function of the first kind, defined as \begin{equation*}
\mathrm{Ai_{3}}(x) = \frac{1}{\pi} \int_0^\infty \cos\left(\alpha x + \frac{\alpha^3}{3} \right) \, d\alpha=\frac{1}{\pi} \sqrt{\frac{x}{3}} K_{1/3}\left( \frac{2}{3} x^{3/2} \right),
\end{equation*}
where $K_{1/3}(\cdot)$ is the modified Bessel function of the second kind of order~$1/3$.

The Airy function in (\ref{Green Airy}) has the following asymptotic behavior \[
\mathrm{Ai_3}(x) \sim
\begin{cases}
\displaystyle \frac{1}{2\sqrt{\pi}x^{1/4}} \exp\left(-\frac{2}{3}x^{3/2}\right), & x \to +\infty, \\[10pt]
\displaystyle \frac{1}{\sqrt{\pi}|x|^{1/4}} \cos\left( \frac{2}{3}|x|^{3/2} - \frac{\pi}{4} \right), & x \to -\infty.
\end{cases}
\]
Therefore, the fundamental solution $G(t,x)$ is exponentially small for large positive $x$ and oscillatory decaying for large negative $x.$
Some further details and properties of Airy functions can be found in~\cite{vallee2004,watson1944,NIST:DLMF} and the references therein.

The expected value of the field $u(t,x)$ in (\ref{Random Airy}) is
\begin{align*}
\mathbb{E}u(t, x) = \frac{1}{\sqrt{\pi}\sqrt[3]{3t}} \int_{\mathbb{R}} M(y)\operatorname{Ai_3}\left( \frac{x - y}{\sqrt[3]{3t}} \right) dy,
\end{align*}
and its covariance function is
\begin{align*}\label{covu}
{\rm{Cov}}(u(t, x), u(t', x')) &= \int_{\mathbb{R}} \cos(\lambda (x - x') + \lambda^3 (t - t')) dF(\lambda).
\end{align*}

Despite an apparent similarity to even-order equations, it will be demonstrated that the analogous limits to Theorem~\ref{the3} do not exist in the odd-order case. The paper proposes an approach based on kernel-smoothed solutions, which results in new limit random fields.

 To formulate some of the following results, we will use the Fox-Wright generalized hypergeometric function ${{}_{p}\Psi_{q}}$  with $p$ numerator and $q$ denominator parameters, see~\cite{pogany2010characteristic}, which is defined as
\begin{equation}\label{psi11}{}_{p}\Psi_{q}\!\left[
\begin{array}{c}
(\alpha_{1},A_{1}),\ldots,(\alpha_{p},A_{p}) \\
(\beta_{1},B_{1}),\ldots,(\beta_{q},B_{q})
\end{array}
; z \right]
= \sum_{n=0}^{\infty}
\frac{ \prod_{j=1}^{p} \Gamma(\alpha_{j} + A_{j} n) }
     { \prod_{k=1}^{q} \Gamma(\beta_{k} + B_{k} n) }
\frac{z^{n}}{n!},
\end{equation}
 where all $A_{j}, \, B_{k}>0$ and $1+\sum_{k=1}^{q}B_{k}-\sum_{j=1}^{p}A_{j}>0$.
\section{The case of even-order equations}\label{sec heat}
This section investigates the even-order heat equations
\begin{equation}\label{classical heat equation1}
   \frac{\partial u(t,x)}{\partial t}=(-1)^{\frac{m}{2}+1}\frac{\partial^{m} u(t,x)}{\partial x^{m}},\quad t>0,\quad x\in \mathbb R,\quad m = 2k,\, k\in \mathbb{N},
\end{equation}
subject to random initial
condition \eqref{initial Airy}.

The random process $\eta (x)=\eta(x,\omega),$ $\omega\in\Omega, x\in \mathbb R,$ is measurable mean-square continuous, with the mean $\mathbb{E}\eta(x)=M(x).$ The process $\xi(x)=\eta(x)-M(x)$ is zero-mean weakly stationary, and $\mathbb{E}|Z(d\lambda)|^{2}=F(d\lambda)$ where $F$ is the spectral measure on the measurable space $(\mathbb{R}, \mathscr{B}(\mathbb{R)})$.

If $F(\cdot)$ is absolutely continuous, then the random process $\eta(x)$ has the spectral representation
\[\eta(x)=M(x)+\int_{\mathbb R} e^{i\lambda x}\sqrt{f(\lambda)}W(d\lambda),\]
where $W(\cdot)$ is the complex white noise Gaussian random measure and $f(\cdot)$ is the spectral density.

By the Bochner-Khinchin theorem, the covariance function of the process $\eta$  has the spectral representation
\begin{equation}\label{covariance function of heat equation1}
    r(x) = {\rm Cov}(\eta(y), \eta(y+x))= {\rm Cov}(\xi(y),\xi(y+x))=\int_{\mathbb R}e^{i\lambda x}F(d\lambda), \quad x\in \mathbb{R}.
\end{equation}

The fundamental solution of the deterministic equation (\ref{classical heat equation1}) can be represented in the following form, see  \cite{orsingher2012probabilistic},
\begin{equation}\label{convolution1}
    u_{m}(t,x)=\frac{1}{2\pi}\int_{\mathbb R}e^{-i\alpha x-\alpha^{m}t}d\alpha=\frac{1}{\sqrt[m]{mt}}g_m\left(\frac{x}{\sqrt[m]{mt}}\right),
\end{equation}
where $g_m(x)=\frac{1}{\pi}\int_{0}^{\infty}\cos(x\alpha)e^{-\frac{\alpha^{m}}{m}}d\alpha$ is the symmetric Lévy stable signed function of order $m$, see \cite{garoni2002mr1893694}.

By the linearity of equation (\ref{classical heat equation1}) the solution to the initial value problem (\ref{classical heat equation1}) and (\ref{initial Airy}) can be written as
\begin{align}\label{solution1}
\hspace*{-0.2cm}u(t,x) &= \int_{\mathbb{R}} \eta(x - y) u_m(t, y) \, dy= \int_{\mathbb{R}}\left(M(x-y)+\int_{\mathbb{R}}e^{i\lambda(x-y)}dZ(\lambda)\right)u_{m}(t,y)dy\nonumber\\&=\int_{\mathbb{R}}M(x-y)u_{m}(t,y)dy+\int_{\mathbb{R}}\int_{\mathbb{R}}e^{i\lambda (x-y)}\sqrt{f(\lambda)}u_{m}(t,y)dW(\lambda)dy\nonumber\\&=\int_{\mathbb{R}}\frac{M(x-y)}{\sqrt[m]{mt}}g_m\left(\frac{y}{\sqrt[m]{mt}} \right)dy+\int_{\mathbb{R}}e^{i\lambda x}\int_{\mathbb{R}}\frac{e^{-i\lambda y}}{\sqrt[m]{mt}}g_m\left(\frac{y}{\sqrt[m]{mt}} \right)dy\, dZ(\lambda).
\end{align}
The order of integration in the above transformations can be changed as $g_m(y)\sqrt{f(\lambda)}\in L_{2}(\mathbb{R}^2).$ The application of the inverse Fourier transform in the internal integral in  \eqref{solution1} and the identity  \eqref{convolution1} give
\begin{equation}\label{*}
    u(t,x)=\int_{\mathbb{R}} \frac{M(x - y)}{\sqrt[m]{mt}} \, g_m\left( \frac{y}{\sqrt[m]{mt}} \right)dy+ \int_{\mathbb{R}}e^{i\lambda x-\lambda^{m}t}dZ(\lambda).
\end{equation}

It follows from (\ref{solution1}) and (\ref{*}) that
\begin{align*}
 \mathbb{E}u(t,x)= \int_{\mathbb{R}} \frac{M(y)}{\sqrt[m]{mt}} \, g_m\left( \frac{x-y}{\sqrt[m]{mt}} \right)dy,
\end{align*}
and the covariance function of $u(t,x)$ equals
\begin{align}\label{covariance f1}
   {\rm Cov}(u(t,x),u(t',x'))& =\int_{\mathbb R}e^{i \lambda (x-x')-\lambda^{m}(t+t')}F(d\lambda)\nonumber\\
   &  = \int_{\mathbb R}\cos(\lambda (x-x'))e^{-\lambda^{m}(t+t')}F(d\lambda).
\end{align}

Formula (\ref{covariance f1}) shows that the random field $u(t,x)$ is stationary in space but not stationary in time.

We generalize Assumption~\ref{AssA} as follows.
\begin{assumption}\label{AssB}
\label{D'''} The covariance function {\rm (\ref{covariance function of heat equation1})} is of the form
\begin{align}\label{covariance of our problem11}
    &r(x)=\sum_{j=0}^{n}\frac{\cos(w_{j}x)}{(1+x^{2})^{\kappa_j/2}}A_{j},\quad x\in\mathbb R,
\end{align}
where $\sum_{j=0}^{n}A_{j}=1, \, w_{0}=0,$ $ w_{j}\neq 0,\,\kappa_j\in(0,1),\, j=0,\dots,n.    $
\end{assumption}

The covariance function in \eqref{covariance of our problem11} is non-integrable and has an oscillating behavior, which corresponds to the cyclic long-range dependence scenario.

It follows from (\ref{covariance of our problem11}), that the corresponding spectral density has the representation
\begin{align}\label{spectral01}
  f(\lambda):=\sum_{j=0}^{n}\frac{c_1(\kappa_j)}{2}A_{j} \left(\frac{K_{\frac{\kappa_j-1}{2}}\left(|\lambda+w_j| \right)}{|\lambda+w_j|^{\frac{1-\kappa_j}{2}}}+ \frac{K_{\frac{\kappa_j-1}{2}}\left(|\lambda-w_j| \right)}{|\lambda-w_j|^{\frac{1-\kappa_j}{2}}}\right),
\end{align}
where $c_1(\kappa_j):={2^{\frac{1-\kappa_j}{2}}}/\left({\sqrt{\pi}{\Gamma\left({\kappa_j}/{2} \right)}}\right),\, j=0,\dots,n$, and $K_\nu(\cdot)$ is the modified Bessel function of the second kind  \[K_\nu(z)=\frac{1}{2}\int_{0}^{\infty}s^{\nu-1}\exp\left(-\frac{1}{2}\left(s+\frac{1}{2} \right)z\right)ds,\quad z\geq 0,\quad \nu\in \mathbb {R}.\]
Also note that $K_{-\nu}(z)=K_{\nu}(z)$, see \cite{NIST:DLMF}.

The spectral density $f(\lambda)$ has singularities at the points $w_j,$ $j=0,...,n.$

The spectral density $f(\cdot)$ is an even function. We select the Gaussian random measure $W(\cdot)$ to be symmetric around the origin to make all random processes and fields considered hereafter real-valued.

\begin{example}\label{ex1}
Figure~{\rm{\ref{fig1}}} shows the covariance function {\rm(\ref{covariance of our problem11})} and the spectral density~{\rm(\ref{spectral01})} for the case of $n=3,$ $A_j=0.20,\,0.20,\, 0.35,\, 0.25,$ $j=0,1,2,3$, with the corresponding $\kappa_{j}$ and $w_{j}$ taking the values $0.2,\,0.4,\, 0.6,\, 0.8$ and $0,\, 0.8,\,  1.2,\,2.0$, respectively.
\begin{figure}[htbp]
\centering
\begin{subfigure}{0.47\textwidth}
    \centering
    \includegraphics[width=\linewidth]{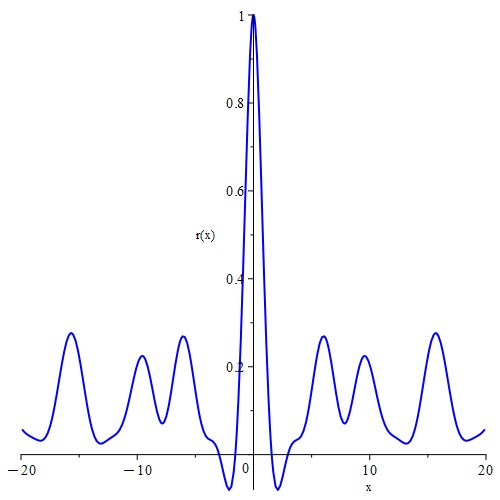} 
    \caption{Covariance function}
\end{subfigure}\hfill
\begin{subfigure}{0.47\textwidth}
    \centering
    \includegraphics[width=\linewidth]{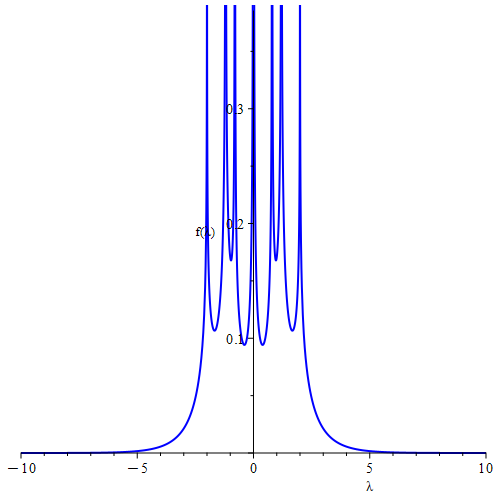} 
    \caption{Spectral density}
\end{subfigure}
\caption{Covariance and spectral structure of the initial condition $\eta(x)$}
\label{fig1}
\end{figure}
\end{example}
The following lemma combines results on the spectral density and its asymptotic behavior, established in \cite{Anh2002RenormalizationAH, alghamdi2024, anh2001}.
\begin{lemma}\label{lem11}
The summands in the spectral density $f(\cdot)$ have the following form
\begin{align}\label{lem}
    &f_{\kappa,w}(\lambda) =  \frac{c_2(\kappa)}{2}A\left(\frac{1-\theta_{\kappa}(|\lambda+w|)}{|\lambda+w|^{1-\kappa}}+\frac{1-\theta_{\kappa}(|\lambda-w|)}{|\lambda-w|^{1-\kappa}}\right),
\end{align}\label{spectral1}
where $c_{2}(\kappa):=\left(2\Gamma(\kappa)\cos\left({\kappa\pi}/{2}\right)\right)^{-1}$ and $\theta_{\kappa}(\cdot)$ does not depend on $w$.

For any $\kappa \in (0,1)$, the function $\theta_\kappa(u)$ is bounded and satisfies $|\theta_{\kappa}(u)| \leq 1$ for all $u \in [0, \infty)$.  Moreover, the following asymptotic expansion near the origin holds
\begin{align*}
    &\theta_{\kappa}(|u|) = \frac{\Gamma\left(\frac{\kappa + 1}{2} \right)}{2^{1 - \kappa} \, \Gamma\left(\frac{3 - \kappa}{2} \right)} \left| u \right|^{1 - \kappa}
- \frac{|u|^2}{2(\kappa + 1)}  + o(|u|^2),\quad  |u| \to 0.
\end{align*}
When $|\lambda| \to \infty$, the following asymptotic behavior holds
\begin{align*}
   f_{\kappa,w}(\lambda)\sim  \frac{\sqrt{\pi}c_{1}(\kappa) (e^{w}+e^{-w})}{2^{3/2}|\lambda|^{1-\kappa/2}}
e^{-|\lambda|}.
\end{align*}\label{condition21}
\end{lemma}
\begin{Remark}
Lemma~{\rm\ref{lem11}} shows that the spectral density has a singularity at $0$ and singularities at the points $\pm w_j$, $j=0,...,n,$ with power-law types $|\lambda\pm w_{j}|^{\kappa_{j}-1}$. It also provides the rate at which deviations from these power-laws occur as $\lambda \to \pm w_j$.
\end{Remark}
\begin{theorem}\label{the31}
Consider a random field \( u(t,x), \ t>0, \ x\in \mathbb{R} \), defined by~{\rm(\ref{solution1})} with \( \eta(x), \ x\in \mathbb{R} \), which has a covariance function satisfying Assumption~{\rm \ref{AssB}} with~$A_0=0$. If \( \varepsilon \to 0 \), then the finite-dimensional distributions of the random~fields
    \begin{equation*}
        U_{\varepsilon}(t,x)=\varepsilon^{-1/4}\,\left(u\left(\frac{t}{\sqrt{\varepsilon^{m}}}, \frac{x}{\sqrt{\varepsilon}}\right)-Eu\left(\frac{t}{\sqrt{\varepsilon^{m}}}, \frac{x}{\sqrt{\varepsilon}}\right)\right)
    \end{equation*}
    converge weakly to the finite-dimensional distributions of the zero-mean Gaussian random field
   \begin{equation}\label{*11}
       U_{0}(t,x)=\sqrt{\sum_{j=1}^{n}c_{1}(\kappa_{j})A_{j}K_{\frac{\kappa_{j}-1}{2}}(|w_{j}|)|w_{j}|^{\frac{\kappa_{j}-1}{2}}}\int_{\mathbb R}e^{i\lambda x- \lambda^{m}t}W(d\lambda)
   \end{equation}
 with the covariance function
 \begin{align}\label{cov31}
    & {\rm Cov}( U_{0}(t,x), U_{0}(t',x')) = \sum_{j=1}^{n}c_{1}(\kappa_{j})A_{j}\frac{K_{\frac{\kappa_{j}-1}{2}}(|w_{j}|)}{|w_{j}|^{\frac{1-\kappa_{j}}{2}}}\int_{\mathbb R}\cos\left(\lambda(x-x')\right)e^{-(t+t')\lambda^{m}}d\lambda\nonumber\\&=\sum_{j=1}^{n}c_{1}(\kappa_{j})A_{j}\frac{K_{\frac{\kappa_{j}-1}{2}}(|w_{j}|)}{|w_{j}|^{\frac{1-\kappa_{j}}{2}}} \frac{2\sqrt{\pi}}{m(t+t')^{1/m}}{{}_{1}\Psi_{1}}\,
\left[
\begin{array}{@{}l@{}}
(1/m,2/m)\\
(1/2,1)
\end{array}
\,; -\frac{(x-x')^{2}}{4(t+t')^{2/m}}
\right],
 \end{align}
 where ${{}_{p}\Psi_{q}}$ is the Fox-Wright generalized hypergeometric function with $p$ numerator and $q$ denominator parameters.
\end{theorem}
\begin{proof}
   It holds in the sense of the finite-dimensional distributions that
   \begin{align}\label{thm1}
       U_{\varepsilon}(t,x) &= \frac{1}{\varepsilon^{1/4}}\Bigg(\int_{\mathbb{R}} \frac{M(y)}{\sqrt[m]{mt}} \, g_m\left( \frac{x-y}{\sqrt[m]{mt}} \right)dy+\int_{\mathbb{R}} e^{i\lambda \frac{x}{\sqrt{\varepsilon}} - \lambda^{m}{t}/{\sqrt{\varepsilon^{m}}}}Z(d\lambda)\nonumber\\&-Eu\left(\frac{t}{\sqrt{\varepsilon^{m}}}, \frac{x}{\sqrt{\varepsilon}}\right)\Bigg)= \frac{1}{\varepsilon^{1/4}}\int_{\mathbb{R}} e^{i\lambda \frac{x}{\sqrt{\varepsilon}}-\lambda^{m}{t}/{\sqrt{\varepsilon^{m}}}}\sqrt{f(\lambda)}W(d\lambda).
   \end{align}
   Note, that by \eqref{spectral01} and \eqref{lem}, one obtains
   \begin{equation}\label{note*}
       c_{1}(\kappa_{j})K_{\frac{\kappa_{j}-1}{2}}(|\lambda|)|\lambda|^{\frac{\kappa_{j}-1}{2}}=c_2(\kappa_j)\frac{1-{\theta_{\kappa_j}(|\lambda|)}}{|\lambda|^{1-\kappa_j}}, \, \lambda\in \mathbb{R}.
   \end{equation}

By (\ref{spectral01}), the change of variables \(\lambda = \tilde{\lambda}\sqrt{\varepsilon} \) and using  the scaling property \( W(d(\tilde{\lambda}\varepsilon))\overset{d}{=} \sqrt{\varepsilon} W(d\tilde{\lambda}), \) one obtains
    \begin{align}\label{***}
   &U_{\varepsilon}(t,x) \overset{d}{=}\int_{\mathbb{R}} e^{i\tilde{\lambda}x- \tilde{\lambda}^{m}t}\left(f(\tilde{\lambda}\sqrt{\varepsilon})\right)^{1/2}W(d\tilde{\lambda})
    \overset{d}{=}\int_{\mathbb{R}} e^{i\tilde{\lambda}x- \tilde\lambda^{m}t} \nonumber \\&\sqrt{\sum_{j=1}^{n}\frac{c_2(\kappa_{j})}{2}A_{j}\left(\frac{1-{\theta_{\kappa_j}}(|\tilde{\lambda}\sqrt{\varepsilon}+w_{j}|)}{|\tilde{\lambda}\sqrt{\varepsilon}+w_{j}|^{1-\kappa_{j}}}+\frac{1-{\theta_{\kappa_j}}(|\tilde{\lambda}\sqrt{\varepsilon}-w_{j}|)}{|\tilde{\lambda}\sqrt{\varepsilon}-w_{j}|^{1-\kappa_{j}}}\right)}W(d\tilde{\lambda}).
\end{align}
It follows from (\ref{*11}) and (\ref{***}) that
    \begin{align}\label{131}
&R(t,x) = \mathbb{E}\left(U_{\varepsilon}(t,x)-U_{0}(t,x)\right)^{2} \nonumber \\
&= \mathbb {E}\Bigg(\int_{\mathbb{R}} e^{i{\lambda}x- {\lambda}^{m}t}\sqrt{\sum_{j=1}^{n}\frac{c_2(\kappa_{j})}{2}A_{j}\left(\frac{1-{\theta_{\kappa_j}}(|{\lambda}\sqrt{\varepsilon}+w_{j}|)}{|{\lambda}\sqrt{\varepsilon}+w_{j}|^{1-\kappa_{j}}}+\frac{1-{\theta_{\kappa_j}}(|{\lambda}\sqrt{\varepsilon}-w_{j}|)}{|{\lambda}\sqrt{\varepsilon}-w_{j}|^{1-\kappa_{j}}}\right)}\nonumber \\
&\times W(d{\lambda})  -\sqrt{\sum_{j=1}^{n}\frac{c_2(\kappa_{j})}{|w_{j}|^{1-\kappa_{j}}}(1-{\theta_{\kappa_j}}(|w_{j}|))A_{j}}\int_{\mathbb{R}}e^{i\lambda x- \lambda^{m}t}W(d{\lambda})\Bigg)^{2}\nonumber\\
& = \int_{\mathbb{R}} e^{-2 {\lambda}^{m}t} \left(\sqrt{Q_{\varepsilon}(\lambda)}-\sqrt{\sum_{j=1}^{n}\frac{c_2(\kappa_{j})}{|w_{j}|^{1-\kappa_{i}}}(1-{\theta_{\kappa_j}}(|w_{j}|))A_{j}}\right)^{2}d\lambda,
\end{align}
where $Q_{\varepsilon}({\lambda}):=\sum_{j=1}^{n}\frac{c_2(\kappa_{j})}{2}A_{j}\left(\frac{1-{\theta_{\kappa_j}}(|{\lambda}\sqrt{\varepsilon}+w_{j}|)}{|{\lambda}\sqrt{\varepsilon}+w_{j}|^{1-\kappa_{j}}}+\frac{1-{\theta_{\kappa_j}}(|{\lambda}\sqrt{\varepsilon}-w_{j}|)}{|{\lambda}\sqrt{\varepsilon}-w_{j}|^{1-\kappa_{j}}}\right).$

Note that \begin{align}\label{Qep}
    \lim_{\varepsilon\to 0}Q_{\varepsilon}(\lambda)&=\sum_{j=1}^{n}A_{j}c_2(\kappa_{j})\frac{1-{\theta_{\kappa_j}}(|w_{j}|)}{|w_{j}|^{1-\kappa_{j}}}=f(0).
\end{align}
Therefore, the integral in (\ref{131}) converges pointwise to $0$.

As it holds
$${\left(\sqrt{Q_{\varepsilon}(\lambda)}-\sqrt{\sum_{j=1}^{n}\frac{c_2(\kappa_{j})}{|w_{j}|^{1-\kappa_{j}}}(1-{\theta_{\kappa_j}}(|w_{j}|))A_{j}}\right)^{2} \leq Q_{\varepsilon}(\lambda)+\sum_{j=1}^{n}\frac{c_2(\kappa_{j})}{|w_{j}|^{1-\kappa_{j}}}(1-{\theta_{\kappa_j}}(|w_{j}|))A_{j}},$$
one may apply the generalized Lebesgue's dominated convergence theorem  \cite[p.59]{Folland1999}. To justify its conditions, one needs to show that the limit
\begin{equation}\label{m1}
\begin{split}
&\lim_{\varepsilon\to 0}\int_{\mathbb{R}}e^{-2\lambda^{m}t}
\left(Q_{\varepsilon}(\lambda)+\sum_{j=1}^{n}
\frac{c_{2}(\kappa_{j})(1-{\theta_{\kappa_j}}(|w_{j}|))}
{|w_{j}|^{1-\kappa_{j}}}A_{j} \right)d\lambda\\
&=\int_{\mathbb{R}}e^{-2\lambda^{m}t}
\lim_{\varepsilon\to 0}\left(Q_{\varepsilon}(\lambda)+\sum_{j=1}^{n}
\frac{c_{2}(\kappa_{j})(1-{\theta_{\kappa_j}}(|w_{j}|))}
{|w_{j}|^{1-\kappa_{j}}}A_{j} \right)d\lambda
< \infty.
\end{split}
\end{equation}

 The finiteness of the last integral in (\ref{m1}) follows from (\ref{Qep}) and the boundedness of the spectral density at zero
 \begin{align*}
\int_{\mathbb{R}}e^{-2\lambda^{m}t}\lim_{\varepsilon\to 0}\left(Q_{\varepsilon}(\lambda)+\sum_{j=1}^{n}\frac{c_{2}(\kappa_{j})(1-{\theta_{\kappa_j}}(|w_{j}|))}{|w_{j}|^{1-\kappa_{i}}}A_{j} \right)d\lambda=2f(0)\int_{\mathbb{R}}e^{-2\lambda^{m}t}d\lambda <\infty.
 \end{align*}
Now, let us consider the first integral in (\ref{m1})
\begin{align}\label{int21}
&\int_{\mathbb{R}}e^{-2\lambda^{m}t}\left(Q_{\varepsilon}(\lambda)+\sum_{j=1}^{n}\frac{c_{2}(\kappa_{j})(1-{\theta_{\kappa_j}}(|w_{j}|))}{|w_{j}|^{1-\kappa_{j}}}A_{j} \right)d\lambda\nonumber\\&=\int_{\mathbb{R}}e^{-2\lambda^{m}t}\sum_{j=1}^{n}A_{j}\frac{c_{2}(\kappa_{j})(1-{\theta_{\kappa_j}}(|\lambda\sqrt{\varepsilon}+w_{j}|))}{2|\lambda\sqrt{\varepsilon}+w_{j}|^{1-\kappa_{j}}}\,d\lambda\nonumber\\&+\int_{\mathbb{R}}e^{-2\lambda^{m}t}\sum_{j=1}^{n}A_{j}\frac{c_{2}(\kappa_{j})(1-{\theta_{\kappa_j}}(|\lambda\sqrt{\varepsilon}-w_{j}|))}{2|\lambda\sqrt{\varepsilon}-w_{j}|^{1-\kappa_{j}}}\, d\lambda\nonumber\\&+\int_{\mathbb{R}}e^{-2\lambda^{m}t}\sum_{j=1}^{n}A_{j}c_{2}(\kappa_{j})\frac{(1-{\theta_{\kappa_j}}(|w_{j}|))}{|w_{j}|^{1-\kappa_{j}}} \,d\lambda\nonumber\\& \quad =I_{1}(\varepsilon)+I_{2}(\varepsilon)+\int_{\mathbb{R}}e^{-2\lambda^{m}t}\sum_{j=1}^{n}A_{j} c_{2}(\kappa_{j})\frac{(1-{\theta_{\kappa_j}}(|w_{j}|))}{|w_{j}|^{1-\kappa_{j}}}\, d\lambda.
\end{align}
Due to the symmetry of the integrands in $I_1$ and $I_2$, we will study only the integral $I_2(\varepsilon).$ Let us introduce the change of variables $u:={\lambda}\sqrt{\varepsilon}$ and split the integration  in $I_2(\varepsilon)$ into two regions, $|u| \leq \varepsilon^{1/2-\delta}$ and $|u| > \varepsilon^{1/2-\delta},$ $ \delta\in (0,1/2)$.
\begin{align}\label{I2_1}
    I_2(\varepsilon) = \varepsilon^{-1/2}\int_{|u| \leq \varepsilon^{1/2-\delta}} e^{-2\frac{u^{m}}{{\varepsilon^{m/2}}}t} \sum_{j=1}^{n}A_{j}\frac{c_{2}(\kappa_{j})(1-{\theta_{\kappa_j}}(|u-w_{j}|))}{2|u-w_{j}|^{1-\kappa_{j}}}\, du \nonumber\\ + \,\varepsilon^{-1/2}\int_{|u| > \varepsilon^{1/2-\delta}} e^{ -2\frac{u^{m}}{{\varepsilon^{m/2}}}t} \sum_{j=1}^{n}A_{j}\frac{c_{2}(\kappa_{j})(1-{\theta_{\kappa_j}}(|u-w_{j}|))}{2|u-w_{j}|^{1-\kappa_{j}}}\, du.
\end{align}
For sufficiently small $\varepsilon,$ $(1-{\theta_{\kappa_j}}(|u-w_{j}|))/{|u-w_{j}|^{1-\kappa_{j}}}$ is an increasing function of $u$ on the interval $[0,\varepsilon^{{1/2}-\delta}]$. Hence, on this interval, $({1-{\theta_{\kappa_j}}(|w_{j}|)})/{|w_{j}|^{1-\kappa_{j}}}$ and $({1-{\theta_{\kappa_j}}(|\varepsilon^{1/2-\delta}-w_{j}|)})/{|\varepsilon^{1/2-\delta}-w_{j}|^{1-\kappa_{j}}}$ are its smallest and largest values, respectively.
Therefore, the first integral in~(\ref{I2_1}) can be bounded  as
\begin{align} \label{bound11}
&\sum_{j=1}^{n}A_{j}\frac{c_{2}(\kappa_{j})(1-{\theta_{\kappa_j}}(|w_{j}|))}{2|w_{j}|^{1-\kappa_{j}}}\int_{|u| \leq \varepsilon^{1/2-\delta}} \frac{e^{-2\frac{u^m}{\varepsilon^{m/2}}t}}{\varepsilon^{1/2}}\, du \nonumber\\\quad&\quad\le\varepsilon^{-1/2}\int_{|u| \leq \varepsilon^{1/2-\delta}} e^{ -2\frac{u^m}{{\varepsilon^{m/2}}}t} \,\sum_{j=1}^{n}A_{j} \frac{c_{2}(\kappa_{j})(1-{\theta_{\kappa_j}}(|u-w_{j}|))}{2|u-w_{j}|^{1-\kappa_{j}}}\, du \nonumber\\\quad& \quad\leq \sum_{j=1}^{n}A_{j}\frac{c_{2}(\kappa_{j})(1-{\theta_{\kappa_j}}(|\varepsilon^{1/2-\delta}-w_{j}|))}{2|\varepsilon^{1/2-\delta}-w_{j}|^{1-\kappa_{j}}}\int_{|u| \leq \varepsilon^{1/2-\delta}} \frac{e^{-2\frac{u^m}{\varepsilon^{m/2}}t}}{\varepsilon^{1/2}}\, du.\end{align}
It follows from \eqref{bound11} and
\[\lim_{\varepsilon\to 0}\int_{|u| \leq \varepsilon^{1/2-\delta}} \frac{e^{-2\frac{u^m}{\varepsilon^{m/2}}t}}{\varepsilon^{1/2}}du=\int_{|u| \leq \varepsilon^{-\delta}} {e^{ -2{u^{m}t}}}du=\int_{\mathbb{R}}e^{-2u^{m}t}du\]
 that
\begin{equation}\label{Q011}
\lim_{\varepsilon \to 0} \varepsilon^{-1/2}\int_{|u| \leq \varepsilon^{1/2-\delta}} e^{-2\frac{u^m}{\varepsilon^{m/2}}t}\sum_{j=1}^{n}
\frac{c_{2}(\kappa_{j})(1-{\theta_{\kappa_j}}(|u-w_{j}|))}{2|u-w_{j}|^{1-\kappa_{j}}}A_{j} du = \frac{f(0)}{2}\int_{\mathbb{R}} e^{-2{\lambda}^{m}t}d{\lambda}.
\end{equation}
For $t>0$, the last integral in (\ref{I2_1}) can be estimated as
\begin{align}\label{estint221}
0&<\varepsilon^{-1/2}\int_{|u| > \varepsilon^{1/2-\delta}} e^{-2\frac{u^m}{\varepsilon^{m/2}}t} \,\sum_{j=1}^{n}A_{i}\frac{c_{2}(\kappa_{j})(1-{\theta_{\kappa_j}}(|u-w_{j}|))}{2|u-w_{i}|^{1-\kappa_{j}}} \,du\nonumber\\&\le\varepsilon^{-1/2} e^\frac{-2{\left(\varepsilon^{1/2-\delta}\right)^{m}t}}{\varepsilon^{m/2}}\int_{|u| > \varepsilon^{1/2-\delta}} \,\sum_{j=1}^{n}A_{i}\frac{c_{2}(\kappa_{j})(1-{\theta_{\kappa_j}}(|u-w_{j}|))}{2|u-w_{j}|^{1-\kappa_{j}}}\, du\nonumber\\&\leq \frac{e^{-2\varepsilon^{-m\delta}t}}{\varepsilon^{1/2}}\int_{\mathbb{R}} \,\sum_{j=1}^{n}A_{i}\frac{c_{2}(\kappa_{j})(1-{\theta_{\kappa_j}}(|u-w_{j}|))}{2|u-w_{j}|^{1-\kappa_{j}}}\, du\to 0,\quad \mbox{when}\quad \varepsilon\to 0,
\end{align}
as due to the integrability of the spectral density, the last integral in (\ref{estint221}) is finite.

Hence, equality in (\ref{m1}) follows from the results (\ref{int21}), (\ref{I2_1}), (\ref{Q011}) and~(\ref{estint221}).
Therefore, by applying the generalized dominated convergence theorem, $\lim_{\varepsilon \to 0}E{\vert U_{\varepsilon}(t,x)-U_{0}(t,x)\vert}^{2}=0,$ which implies the convergence of the corresponding finite-dimensional distributions.

Finally, the formula {(\rm\ref{cov31})} follows from
\begin{align*}
{\rm Cov}&(U_{0}(t,x), U_{0}(t',x')) \\&= \mathbb{E} \left( \int_{\mathbb{R}} \int_{\mathbb{R}} \sum_{j=1}^{n} \frac{c_{2}(\kappa_{j})}{|w_j|^{1-\kappa_{j}}}A_j (1 - {\theta_{\kappa_j}}(|w_j|))
\, e^{i\lambda_{1} x - t \lambda_{1}^{m}}
\, e^{-i\lambda_{2} x' - t' \lambda_{2}^{m}} \, W(d\lambda_{1}) W(d\lambda_{2}) \right) \\
&= \sum_{j=1}^{n} \frac{c_{2}(\kappa_{j})}{|w_j|^{1-\kappa_{j}}}A_j (1 - {\theta_{\kappa_j}}(|w_j|))
\int_{\mathbb{R}} e^{i\lambda(x - x')} e^{-(t + t') \lambda^{m}} \, d\lambda = |(t+t')^{1/m}\lambda=\tilde{\lambda}|\\&=\sum_{j=1}^{n} \frac{c_{2}(\kappa_{j})}{|w_j|^{1-\kappa_{j}}}A_j (1 - {\theta_{\kappa_j}}(|w_j|))(t+t')^{-1/m}\int_{\mathbb{R}}e^{i\tilde{\lambda}{(x-x')}/{(t+t')^{1/m}}}e^{-\tilde{\lambda^{m}}}d\tilde{\lambda}\\&=\sum_{j=1}^{n} \frac{c_{2}(\kappa_{j})}{|w_j|^{1-\kappa_{j}}}A_j (1 - {\theta_{\kappa_j}}(|w_j|))(t+t')^{-1/m}\,\mathcal{F}\left( \frac{x-x'}{(t+t')^{1/m}}\right),
\end{align*}
where $\mathcal{F}(\cdot)$ is the inverse Fourier transform of the function $e^{-\lambda^{m}}$. Notice that $e^{-\lambda^{m}}$ equals to $2m^{-1}\Gamma(1/m)f_{m}(0,1;\lambda)$, where $f_{m}(\mu,\sigma;\lambda)$ is the probability density of the generalized normal distribution, see \cite{pogany2010characteristic}. Therefore, by Theorem 2.1 in \cite{pogany2010characteristic},  \[\mathcal{F}\left(\frac{x-x'}{(t+t')^{1/m}}\right)=2m^{-1}\Gamma(1/m)\frac{\sqrt{\pi}}{\Gamma(1/m)}{{}_{1}\Psi_{1}}\,
\left[
\begin{array}{@{}l@{}}
(1/m,2/m)\\
(1/2,1)
\end{array}
\,; -\frac{(x-x')^{2}}{4(t+t')^{2/m}}
\right],\]
and by \eqref{note*} the covariance function of $U_{0}(t,x)$ is equal to
 \begin{align*}
     {\rm Cov}( U_{0}(t,x), U_{0}(t',x'))& =\sum_{j=1}^{n}c_{1}(\kappa_{j})A_{j}K_{\frac{\kappa_{j}-1}{2}}(|w_{j}|)|w_{j}|^{\frac{\kappa_{j}-1}{2}}\\&\times \frac{2\sqrt{\pi}}{m(t+t')^{1/m}}{{}_{1}\Psi_{1}}\,
\left[
\begin{array}{@{}l@{}}
(1/m,2/m)\\
(1/2,1)
\end{array}
\,; -\frac{(x-x')^{2}}{4(t+t')^{2/m}}
\right],
 \end{align*}
which completes the proof of the theorem.
\end{proof}
\begin{Remark}
Theorem~{\rm\ref{the3}} is a special case of Theorem~{\rm\ref{the31}} for $m=2$ and $n=1.$
\end{Remark}

\begin{example}\label{ex2} The values of $\kappa_{j}$ and $w_j$ were chosen to be the same as in Example~{\rm \ref{ex1}}. However, since $A_0=0$ in Theorem~{\rm\ref{the31}}, to ensure that the sum of the $A_j$ values equals~$1,$ $A_1$ was changed to $0.4$.

Figure~{\rm \ref{fig:a}} presents the realization of the random field $U_{0}(t,x)$ for the 4th-order case, i.e., $m=4$. To obtain this realization, the following approximation of the stochastic integral in \eqref{*11} was used
\begin{equation}
       \int_{\mathbb R}e^{i\lambda x- \lambda^{m}t}W(d\lambda)\approx \sum_{j=-N}^{N-1} e^{i\lambda_j x- \lambda_j^{m}\,t}W(\Delta\lambda_j),
   \end{equation}
   where  $\lambda_j\in \{-N\Delta,-(N-1) \Delta,...,-\Delta,0,\Delta,...,N\Delta\},$ $\Delta\lambda_j=\lambda_{j+1}-\lambda_{j}$, $\Delta=0.05,$ and $N=1000$.

   The simulation illustrates the evolution of the field $U_0(t, x)$. For times close to 0, the field appears rough and exhibits rapid spatial oscillations, which arise from the random initial conditions. As time progresses, the field becomes smoother. This behavior is due to the exponential decay factor $e^{-\lambda^4 t}$, which suppresses fluctuations. Figure~{\rm \ref{fig:b}} presents the covariance function $\text{Cov}(U_0(t,x), U_0(t',x'))$, which has its largest values when $x \approx x'$ and $t, t' \approx 0$. The covariance decays rapidly when either the spatial separation $|x - x'|$ or the total time $t + t'$ increases.

Figure~{\rm \ref{fig3}} illustrates the influence of the parameter $m$ on the covariance structure. In Figure~{\rm \ref{fig:c}}, where the spatial separation is fixed at $|x - x'| = 1.5$, the covariance decays gradually as $t + t'$ increases. Larger values of $m$ yield correspondingly larger covariance values. In Figure~{\rm \ref{fig:d}}, with the time fixed at $t + t' = 5$, the covariance exhibits oscillations and quickly vanishes as the spatial separation $|x - x'|$ increases. Higher values of $m$ result in more pronounced oscillatory behavior.
 \begin{figure}[htbp]
    \centering
    \begin{subfigure}{0.5\textwidth}
        \centering
        \includegraphics[width=1.1\linewidth, trim=7cm 1.8cm 2.2cm 5.0cm, clip]{"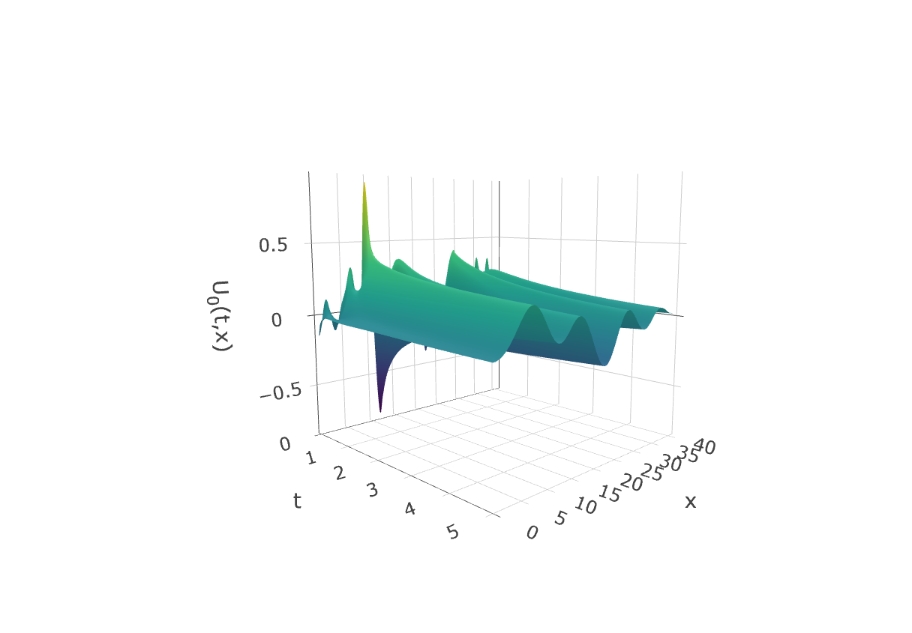"}
        \caption{Realization of $U_{0}(t,x)$}
        \label{fig:a}
    \end{subfigure}\hfill
    \begin{subfigure}{0.5\textwidth}
        \centering
        \includegraphics[width=1.1\linewidth, trim=5.8cm 1.8cm 2cm 4.8cm, clip]{"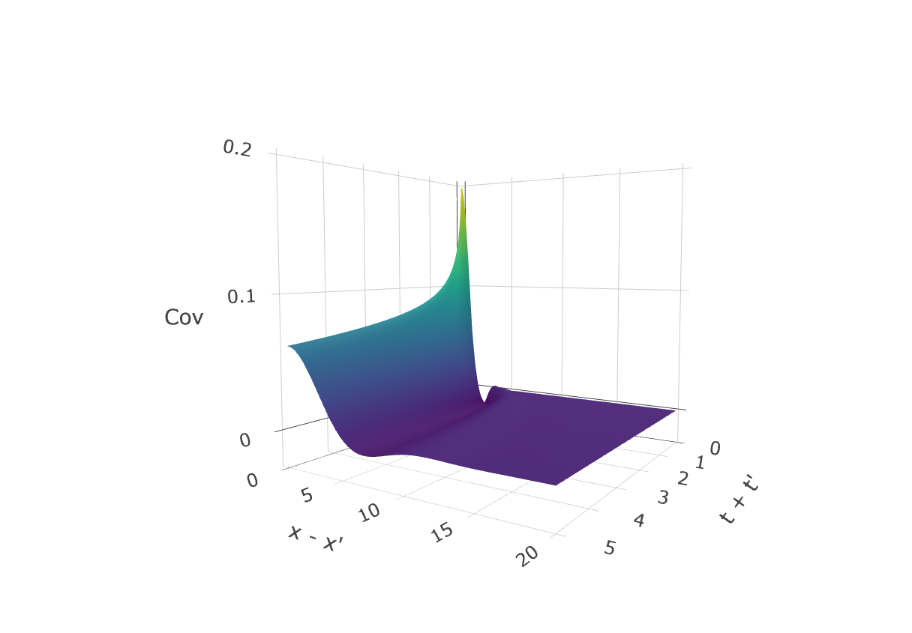"}
        \caption{Covariance function of $U_{0}(t,x)$}
        \label{fig:b}
    \end{subfigure}
    \caption{The case of the 4th-order heat equation with $A_{0}=0$}
    \label{fig2}
\end{figure}

\begin{figure}[htbp]
    \centering
    \begin{subfigure}{0.485\textwidth}
        \centering
        \includegraphics[width=\linewidth]{"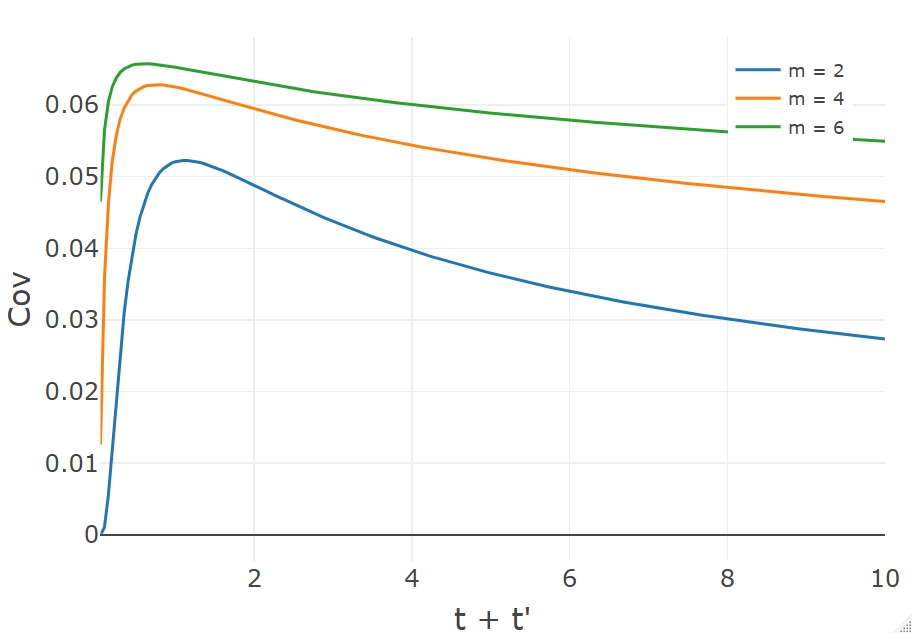"}
        \caption{Temporal covariance function,\,$x-x'=1.5$}
        \label{fig:c}
    \end{subfigure}\hfill
    \begin{subfigure}{0.48\textwidth}
        \centering
        \includegraphics[width=\linewidth]{"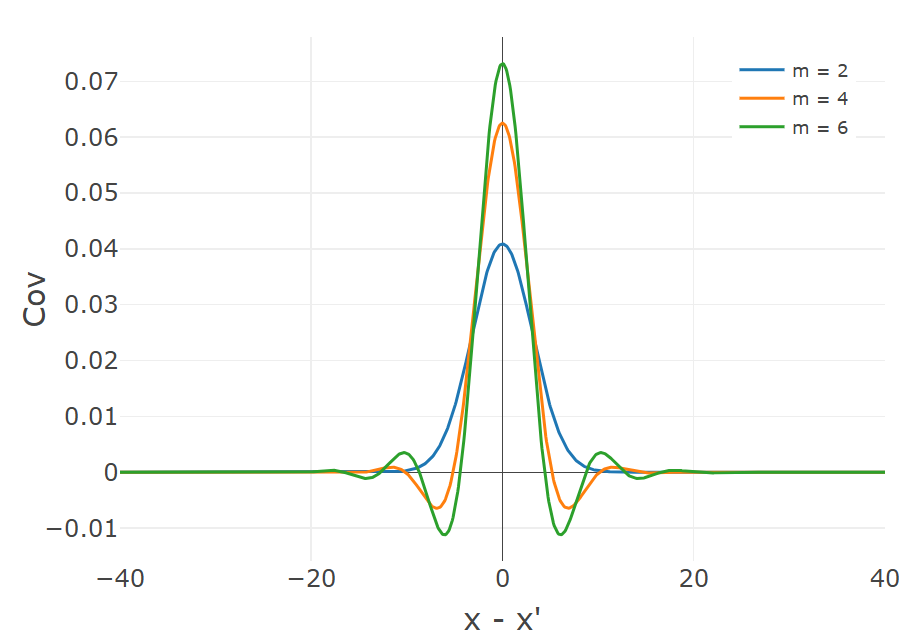"}
        \caption{Spatial covariance function, $t+t'=5$}
        \label{fig:d}
    \end{subfigure}
    \caption{Covariance structures of $U_{0}(t,x)$ for $m$th-order heat equations with $A_{0}=0$}
    \label{fig3}
\end{figure}

\end{example}
\begin{theorem}\label{the32}
Consider a random field \( u(t,x), \ t>0, \ x\in \mathbb{R} \), defined by~{\rm(\ref{solution1})} with \( \eta(x), \ x\in \mathbb{R} \), which has a covariance function satisfying Assumption~{\rm \ref{AssB}} with $A_{0}\neq 0$. If \( \varepsilon \to 0 \), then the finite-dimensional distributions of the random~fields
    \begin{equation*}
        U_{\varepsilon}(t,x)=\varepsilon^{-\kappa_0/4}\,\left(u\left(\frac{t}{\sqrt{\varepsilon^{m}}}, \frac{x}{\sqrt{\varepsilon}}\right)-Eu\left(\frac{t}{\sqrt{\varepsilon^{m}}}, \frac{x}{\sqrt{\varepsilon}}\right)\right)
    \end{equation*}
    converge weakly to the finite-dimensional distributions of the zero-mean Gaussian random field
  \begin{equation}\label{eq:11}
U_{0}(t,x)
= \sqrt{A_{0}\,c_{2}(\kappa_{0})}\,
\int_{\mathbb{R}}
\frac{e^{\,i\lambda x - \lambda^{m} t}}{|\lambda|^{\frac{1-\kappa_{0}}{2}}}\, W(d\lambda)
\end{equation}
 with the covariance function
 \begin{align}\label{cov311}
     {\rm Cov}( U_{0}(t,x), U_{0}(t',x')) &= A_{0}\,c_{2}(\kappa_{0})\int_{\mathbb R}|\lambda|^{\kappa_{0}-1}\cos\left(\lambda(x-x')\right)e^{-(t+t')\lambda^{m}}d\lambda\nonumber\\
     &\hspace{-1cm}=\frac{2\sqrt{\pi}\,A_{0}\,c_{2}(\kappa_{0})}{m(t+t')^{\kappa_{0}/m}}\;
{}_{1}\Psi_{1}\!\left[
\begin{array}{c}
(\kappa_{0}/m,\;2/m)\\[2pt]
(1/2,\;1)
\end{array}
;\;
-\frac{(x-x')^{2}}{4\,(t+t')^{2/m}}
\right].
 \end{align}
\end{theorem}
\begin{proof}
   Analogously to \eqref{thm1}, it holds in the sense of the finite-dimensional distributions that
   \begin{align*}
       U_{\varepsilon}(t,x) = \frac{1}{\varepsilon^{\kappa_0/4}}\int_{\mathbb{R}} e^{i\lambda \frac{x}{\sqrt{\varepsilon}}-\lambda^{m}{t}/{\sqrt{\varepsilon^{m}}}}\sqrt{f(\lambda)}W(d\lambda).
   \end{align*}
By (\ref{spectral01}), the change of variables \(\lambda = \tilde{\lambda}\sqrt{\varepsilon} \) and using  the scaling property \( W(d(\tilde{\lambda}\varepsilon))\overset{d}{=} \sqrt{\varepsilon} W(d\tilde{\lambda}), \) one obtains
\begin{align}\label{*1them}
U_{\varepsilon}(t,x)
&\overset{d}{=}\sqrt{A_{0}\,c_2(\kappa_0)}
\int_{\mathbb{R}}
\frac{e^{i\tilde{\lambda} x - \tilde{\lambda}^{m} t}}{|\tilde{\lambda}|^{\frac{1-\kappa_0}{2}}}
\Bigg(
\sum_{j=1}^{n}
\frac{A_{j}\,c_2(\kappa_{j})}{2\,A_{0}\,c_2(\kappa_0)}\,|\tilde{\lambda}\sqrt{\varepsilon}|^{1-\kappa_0}
\left(
\frac{1-\theta_{\kappa_j}\!\big(|\tilde{\lambda}\sqrt{\varepsilon}+w_{j}|\big)}
     {|\tilde{\lambda}\sqrt{\varepsilon}+w_{j}|^{\,1-\theta_{\kappa_j}}}\right.\nonumber
\\
&\quad +\left.
\frac{1-\theta_{\kappa_j}\!\big(|\tilde{\lambda}\sqrt{\varepsilon}-w_{j}|\big)}
     {|\tilde{\lambda}\sqrt{\varepsilon}-w_{j}|^{\,1-\kappa_{j}}}
\right)
\;+\;
\big(1-\theta_{\kappa_{0}}\!\big(|\tilde{\lambda}\sqrt{\varepsilon}|\big)\big)
\Bigg)^{\!1/2}
\, W(\mathrm{d}\tilde{\lambda}) .
\end{align}

Therefore, it follows from (\ref{eq:11}) and (\ref{*1them}) that
    \begin{align}\label{eqth2}
&R(t,x) = \mathbb{E}\left(U_{\varepsilon}(t,x)-U_{0}(t,x)\right)^{2}=
 \mathbb {E}\Bigg(\sqrt{A_{0}\,c_2(\kappa_0)}\int_{\mathbb{R}}\frac{e^{i{\lambda}x- {\lambda}^{m}t}}{|\lambda|^{\frac{1-\kappa_0}{2}}}  \nonumber \\&\times\Bigg({\sum_{j=1}^{n}\frac{A_{j}\,c_2(\kappa_{j})}{2\,A_{0}\,c_2(\kappa_0)}|{\lambda}\sqrt{\varepsilon}|^{1-\kappa_0}\left(\frac{1-{\theta_{\kappa_{j}}}(|{\lambda}\sqrt{\varepsilon}+w_{j}|)}{|{\lambda}\sqrt{\varepsilon}+w_{j}|^{1-\kappa_{j}}}+\frac{1-{\theta_{\kappa_{j}}}(|{\lambda}\sqrt{\varepsilon}-w_{j}|)}{|{\lambda}\sqrt{\varepsilon}-w_{j}|^{1-\kappa_{j}}}\right)}\nonumber\\&+(1-{\theta_{\kappa_{0}}}(|{\lambda}\sqrt{\varepsilon}|)\Bigg)^{1/2}W(d{\lambda})  -\sqrt{A_{0}\,c_2(\kappa_{0})}\int_{\mathbb{R}}\frac{e^{i\lambda x- \lambda^{m}t}}{|\lambda|^{\frac{1-\kappa_0}{2}}}W(d{\lambda})\Bigg)^{2}\nonumber\\
& = A_{0}\,c_2(\kappa_0)\int_{\mathbb{R}} \frac{e^{- 2\lambda^{m}t}}{|\lambda|^{{1-\kappa_0}}}\left(\sqrt{Q_{\varepsilon}(\lambda)}-1\right)^{2}d\lambda,
\end{align}
where
\begin{align*}
    Q_{\varepsilon}({\lambda}):=&{\sum_{j=1}^{n}\frac{A_{j}\,c_2(\kappa_{j})}{2\,A_{0}\,c_2(\kappa_0)}|{\lambda}\sqrt{\varepsilon}|^{1-\kappa_0}\left(\frac{1-\theta_{j}(|{\lambda}\sqrt{\varepsilon}+w_{j}|)}{|{\lambda}\sqrt{\varepsilon}+w_{j}|^{1-\kappa_{j}}}+\frac{1-\theta_{j}(|{\lambda}\sqrt{\varepsilon}-w_{j}|)}{|{\lambda}\sqrt{\varepsilon}-w_{j}|^{1-\kappa_{j}}}\right)}\\&+(1-\theta_0(|{\lambda}\sqrt{\varepsilon}|).
\end{align*}
Note that $ \lim_{\varepsilon\to 0}Q_{\varepsilon}(\lambda)=1.$
Therefore, the integrand in (\ref{eqth2}) converges pointwise to $0$.
As it holds
$\left(\sqrt{Q_{\varepsilon}(\lambda)}-1\right)^{2} \leq Q_{\varepsilon}(\lambda)+1$,
one may apply the generalized Lebesgue's dominated convergence theorem, see \cite[p.59]{Folland1999}. To justify its conditions, one needs to show that the limit
\begin{align}\label{mm1}
&\lim_{\varepsilon\to 0}A_{0}c_{2}
(\kappa_0)\int_{\mathbb{R}} \frac{e^{-2 \lambda^{m}t}}{|\lambda|^{{1-\kappa_0}}}
\left(Q_{\varepsilon}(\lambda)+1 \right)d\lambda
=A_{0}c_{2}
(\kappa_0)\int_{\mathbb{R}}\frac{e^{- 2\lambda^{m}t}}{|\lambda|^{{1-\kappa_0}}}
\lim_{\varepsilon\to 0}\left(Q_{\varepsilon}(\lambda)+1\right)d\lambda
< \infty.
\end{align}
The finiteness of the last integral in (\ref{mm1}) follows from the limit of $Q_{\varepsilon}(\lambda)$ stated before and the condition $\kappa_{0}\in(0,1)$, as
 \begin{align}\label{finite 1}
&A_{0}c_2(\kappa_0)\int_{\mathbb{R}}\frac{e^{- 2\lambda^{m}t}}{|\lambda|^{{1-\kappa_0}}}\lim_{\varepsilon\to 0}\left(Q_{\varepsilon}(\lambda)+1 \right)d\lambda=2A_{0}c_2(\kappa_0)\int_{\mathbb{R}}\frac{e^{- 2\lambda^{m}t}}{|\lambda|^{{1-\kappa_0}}}d\lambda <\infty.
 \end{align}

Now, the integral on the left-hand side of (\ref{mm1}) can be rewritten as
\begin{align}\label{int212}
&\int_{\mathbb{R}}\frac{e^{-2\lambda^{m}t}}{|\lambda|^{1-\kappa_0}}\left(Q_{\varepsilon}(\lambda)+1\right)d\lambda=\int_{\mathbb{R}}\frac{e^{-2\lambda^{m}t}}{|\lambda|^{1-\kappa_0}}\sum_{j=1}^{n}|\lambda\sqrt{\varepsilon}|^{1-\kappa_0}\frac{A_{j}\,c_{2}(\kappa_{j})(1-{\theta_{\kappa_j}}(|\lambda\sqrt{\varepsilon}+w_{j}|))}{2\,A_{0}\,c_2(\kappa_0)|\lambda\sqrt{\varepsilon}+w_{j}|^{1-\kappa_{j}}}\,d\lambda\nonumber\\& +\int_{\mathbb{R}}\frac{e^{-2\lambda^{m}t}}{|\lambda|^{1-\kappa_0}}\sum_{j=1}^{n}|\lambda\sqrt{\varepsilon}|^{1-\kappa_0}\frac{A_{j}\,c_{2}(\kappa_{j})(1-{\theta_{\kappa_j}}(|\lambda\sqrt{\varepsilon}-w_{j}|))}{2\,A_{0}\,c_{2}(\kappa_0)|\lambda\sqrt{\varepsilon}-w_{j}|^{1-\kappa_{j}}}\, d\lambda \nonumber\\&+\int_{\mathbb{R}}\frac{e^{-2\lambda^{m}t}}{|\lambda|^{1-\kappa_0}}(1-{\theta_{\kappa_{0}}}(|\lambda\sqrt{\varepsilon}|))d\lambda+\int_{\mathbb{R}}\frac{e^{-2\lambda^{m}t}}{|\lambda|^{1-\kappa_0}}d\lambda =I_{1}(\varepsilon)+I_{2}(\varepsilon)+I_{3}(\varepsilon)+I_{4}.
\end{align} Note, that $I_{1}(\varepsilon)+I_{2}(\varepsilon)\to 0$, as $\varepsilon\to 0$, and $I_{4}$ is finite. The integral $I_{3}(\varepsilon)\to I_{4}$, when $\varepsilon\to 0$. It follows from $\lim_{\varepsilon\to 0}\theta_{0}(\varepsilon)=0, \,\theta_{0}(\lambda)\in[0,1]$, and the dominated convergence theorem. Hence, the result in \eqref{mm1} follows from \eqref{finite 1} and \eqref{int212}. By the generalized dominated convergence theorem  $\lim_{\varepsilon \to 0}E{\vert U_{\varepsilon}(t,x)-U_{0}(t,x)\vert}^{2}=0,$ which implies the convergence of finite-dimensional distributions.

Finally, to obtain the formula {(\rm\ref{cov311})} note that
\begin{align*}
&{\rm Cov}(U_{0}(t,x), U_{0}(t',x')) \\&= \mathbb{E} \left( \int_{\mathbb{R}} \int_{\mathbb{R}} A_{0}\,c_{2}(\kappa_{0})|\lambda_1|^{(\kappa_{0}-1)/2}|\lambda_2|^{(\kappa_{0}-1)/2}
\, e^{i\lambda_{1} x - t \lambda_{1}^{m}}
 e^{-i\lambda_{2} x' - t' \lambda_{2}^{m}} \, W(d\lambda_{1}) W(d\lambda_{2}) \right) \\
&= A_{0}\,c_2(\kappa_0)
\int_{\mathbb{R}} |\lambda|^{\kappa_{0}-1}e^{i\lambda(x - x')} e^{-(t + t') \lambda^{m}} \, d\lambda\\&= 2\,A_{0}\,c_{2}(\kappa_{0})\int_{0}^{\infty}\lambda^{\kappa_{0}-1}\cos(\lambda(x-x'))e^{-(t+t')\lambda^{m}}d\lambda.
\end{align*}
By using the change of variables $(t+t')^{1/m}\lambda=\tilde{\lambda}$ and expanding the cosine into Maclaurin series, analogously to \cite{pogany2010characteristic}, one obtains that
\begin{align*}
    {\rm Cov}(U_{0}(t,x), U_{0}(t',x'))&=\frac{2\,A_{0}\,c_{2}(\kappa_{0})}{(t+t')^{1/m}}\int_{0}^{\infty}\left(\frac{\tilde{\lambda}}{(t+t')^{1/m}}\right)^{\kappa_{0}-1}\cos\left(\frac{\tilde{\lambda}(x-x')}{(t+t')^{1/m}}\right)e^{-\tilde{\lambda}^m}d\tilde{\lambda}\\
    &   =\frac{2\,A_{0}\,c_{2}(\kappa_{0})}{(t+t')^{\kappa_{0}/m}}\sum_{k=0}^{\infty}\frac{(-1)^{k}(x-x')^{2k}}{(2k)!(t+t')^{{2k}/{m}}}\int_{0}^{\infty}\tilde{\lambda}^{\kappa_{0}-1+2k}e^{-\tilde{\lambda}^{m}}d\lambda.  \end{align*}
    Let us $u=\tilde{\lambda}^{m}.$ Then, $ \tilde{\lambda}=u^{1/m},$ $d\tilde{\lambda}={m}^{-1}u^{{1/m}-1}du,$ and
   \begin{align*}\label{coeps}  {\rm Cov}(U_{0}(t,x), U_{0}(t',x'))&=\frac{2\,A_{0}\,c_{2}(\kappa_{0})}{m(t
   +t')^{\kappa_{0}/m}}\sum_{k=0}^{\infty}\frac{(-1)^{k}}{(2k)!}\left(\frac{(x-x')}{(t+t')^{1/m}}\right)^{2k}\int_{0}^{\infty}u^{\frac{\kappa_{0}}{m}+\frac{2k}{m}-1}e^{-u}du\nonumber\\
&=\frac{2\,A_{0}\,c_{2}(\kappa_{0})}{m(t+t')^{\kappa_{0}/m}}\sum_{k=0}^{\infty}\frac{\Gamma(\kappa_{0}/m+2k/m)}{\Gamma(2k+1)}\left(-\frac{(x-x')^{2}}{(t+t')^{2/m}} \right)^{k}.
 \end{align*}
The application of the Legendre duplication formula
$
\Gamma(2z)\;=\;{2^{\,2z-1}}\Gamma(z)\,\Gamma\!\left(z+1/2\right)/\sqrt{\pi}$
gives
\[\Gamma(2k+1)=2k\Gamma(2k)=2k\frac{2^{\,2k-1}}{\sqrt{\pi}}\;\Gamma(k)\,\Gamma\!\left(k+1/2\right)=\frac{2^{2k}}{\sqrt{\pi}}\,k!\,\Gamma(k+1/2),\] and by (\ref{psi11}) one obtains
 \begin{align*}
 {\rm Cov}(U_{0}(t,x), U_{0}(t',x'))&=\frac{2\sqrt{\pi}\,A_{0}\,c_{2}(\kappa_{0})}{m(t+t')^{\kappa_{0}/m}} \sum_{k=0}^{\infty}\frac{\Gamma(\kappa_{0}/m+2k/m)}{\Gamma(1/2+k)}\left(-\frac{(x-x')^{2}}{4(t+t')^{2/m}}\right)^{k}\frac{1}{k!}\\&=\frac{2\sqrt{\pi}\,A_{0}\,c_{2}(\kappa_{0})}{m(t+t')^{\kappa_{0}/m}}\;
{}_{1}\Psi_{1}\!\left[
\begin{array}{c}
(\kappa_{0}/m,\;2/m)\\[2pt]
(1/2,\;1)
\end{array}
;\;
-\frac{(x-x')^{2}}{4\,(t+t')^{2/m}}
\right],
\end{align*}
which completes the proof of the theorem.
\end{proof}
\begin{Remark}
    The limit fields $U_0(t,x)$ in Theorems~{\rm\ref{the31}}  and~{\rm\ref{the32}}  are stationary in space, but non-stationary in time.
\end{Remark}

\begin{example}\label{ex3} This example illustrates the limit random field $U_0(t,x)$ from Theorem~{\rm \ref{the32}}. The same values of $A_{j}$, $\kappa_{j}$, and $w_j$ are used as in Example~{\rm \ref{ex1}}. Figure~{\rm \ref{fig:a1}} shows a realization of the random field $U_0(t,x)$. The simulated field is rather similar to that in Example~{\rm \ref{ex2}}. However, it exhibits a greater scale of variation. This seems to be the result of the additional multiplicative term $|\lambda|^{\frac{1-\kappa_{0}}{2}}$ in the spectrum, which has the singularity at the origin.
Figure~{\rm \ref{fig:b1}} shows a slower spatial decay of the correlations compared to the corresponding plot in  Example~{\rm \ref{ex2}.}

Figures~{\rm \ref{fig:c1}} and~{\rm \ref{fig:d1}} illustrate the behavior of covariance functions for different values of $m.$ Both graphs indicate larger covariance values than those in Example~{\rm \ref{ex2}}. The first plot exhibits a time dependence structure that is similar to that in Example~{\rm \ref{ex2}}, but with a faster decay in the case $m=2$. However, there are several differences in the space dependencies. They persist over longer distances compared to those in Figure~{\rm \ref{fig:d}}, where they vanish more rapidly and exhibit oscillating behavior.
\begin{figure}[htbp]
\centering
\begin{subfigure}{0.47\textwidth}
    \centering
    \includegraphics[width=1.1\linewidth,trim=5.5cm 1.6cm 4cm 5.1cm,clip]{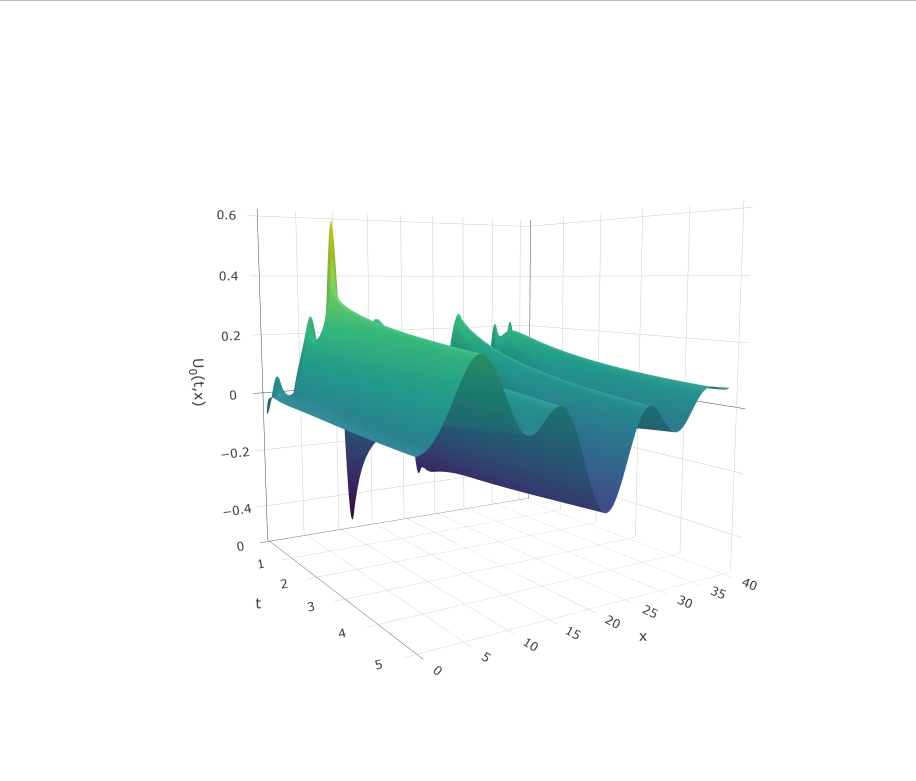}
    \caption{Realization of $U_{0}(t,x)$}
    \label{fig:a1}
\end{subfigure}\hfill
\begin{subfigure}{0.47\textwidth}
    \centering
    \includegraphics[width=1.1\linewidth,trim=5.3cm 1.6cm 5cm 4cm,clip]{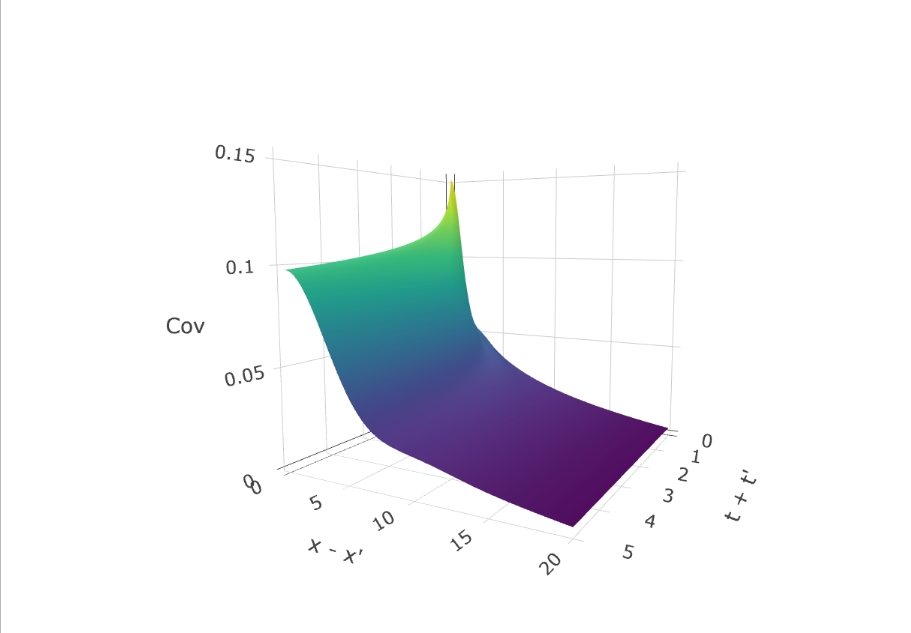}
    \caption{Covariance function of $U_{0}(t,x)$}
    \label{fig:b1}
\end{subfigure}
\caption{The case of the 4th-order heat equation with $A_{0}\neq0$}
\label{fig4}
\end{figure}

\begin{figure}[htbp]
\centering
\begin{subfigure}{0.485\textwidth}
    \centering
    \includegraphics[width=\linewidth]{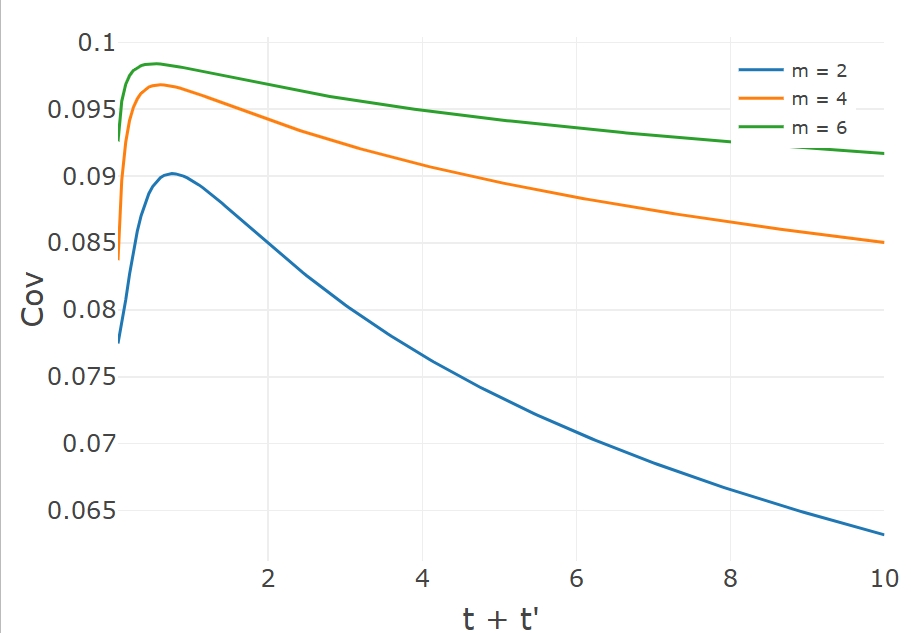}
    \caption{Temporal covariance function,\,$x-x'=1.5$}
    \label{fig:c1}
\end{subfigure}\hfill
\begin{subfigure}{0.48\textwidth}
    \centering
    \includegraphics[width=\linewidth]{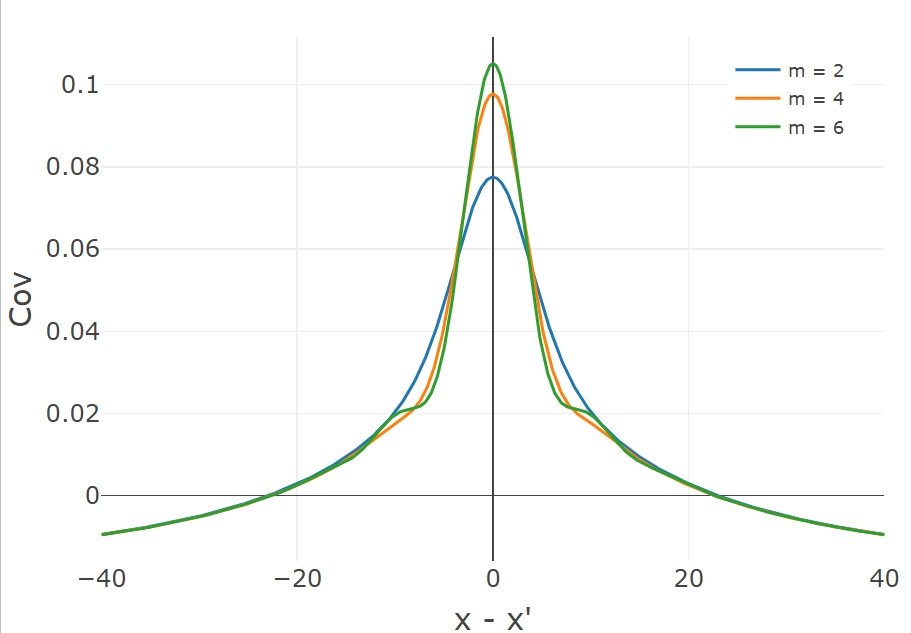}
    \caption{Spatial covariance function, $t+t'=5$}
    \label{fig:d1}
\end{subfigure}
\caption{Covariance structures of $U_{0}(t,x)$ for $m$th-order heat equations with $A_{0}\neq0$}
\label{fig5}
\end{figure}

\end{example}
\section{Limit theorems for odd order equations}\label{secodd}
In this section, we study the limiting behavior of solutions to odd-order heat-type PDEs with cyclic long-range dependent initial conditions. Unlike even-order equations, fundamental solutions of odd-order equations involve highly asymmetric and oscillatory Airy functions, which prevent the application of the previous approach to prove scaling limit theorems.

Let us consider the odd-order heat equations
\begin{equation}\label{eq15}
    \frac{\partial u(t, x)}{\partial t} = \mu \frac{\partial^m u(t, x)}{\partial x^m}, \quad t > 0,\ x \in \mathbb{R},\, \mu=\pm1,\quad m = 2k + 1, \, k\in \mathbb{N},
\end{equation}
with random initial condition \eqref{initial Airy}.
The deterministic solution to equation \eqref{eq15} is, see \cite{orsingher2012probabilistic},
\begin{equation*}
    u_{m}(t,x)=\frac{1}{2\pi}\int_{\mathbb{R}}e^{-i\lambda x-i\mu (-1)^{\frac{m-1}{2}}t\lambda^{m}}d\lambda=\frac{1}{\sqrt[m]{mt}}g_{m}\left( \frac{x}{\sqrt[m]{mt}}\right),
\end{equation*}
where $g_{m}(x)=\frac{1}{\pi}\int_{\mathbb{R}}\cos\left(\alpha x-\mu (-1)^{\frac{m-1}{2}}\frac{\alpha^m}{m}\right)d\alpha, \, x\in \mathbb{R}$.

\begin{Remark}
Note, that for $\mu=(-1)^{k}$ the function $g_{m}(\cdot)$ coincides with the $m$-order Airy function ${\rm Ai}_{m}(x)=\frac{1}{\pi}\int_{\mathbb{R}}\cos(\alpha x+{\alpha^m}/{m})\,d\alpha,$  see  {\rm\cite{marchione2024stable} }and {\rm \cite{Ansari}}.
\end{Remark}
Hence, analogously to \eqref{*}, the solution field is given by
\begin{align}\label{+}
     u(t, x) &= \int_{\mathbb{R}}\frac{M(x-y)}{\sqrt[3]{mt}}g_{m}\left(\frac{y}{\sqrt[m]{mt}} \right)+\frac{1}{2\pi}\int_{\mathbb{R}} e^{i \lambda x + i \mu (-1)^{\frac{m-1}{2}}\lambda^m t} Z(d\lambda)\nonumber\\&=\int_{\mathbb{R}}\frac{M(x-y)}{\sqrt[3]{mt}}g_{m}\left(\frac{y}{\sqrt[m]{mt}} \right)+ \frac{1}{2\pi} \int_{\mathbb{R}} e^{i\lambda x + i\mu (-1)^{\frac{m-1}{2}}\lambda^m t} \sqrt{f(\lambda)} \, W(d\lambda).
\end{align}

\begin{Remark}
For the odd-order heat equation, the same method of proof as in the even-order case can not be used. Indeed, if one were to formally follow  the proof of Theorem~{\rm{\ref{the31}}}, not justifying dominated convergence, the limit field would be of the form \[\tilde{U}_{0,1}(t,x)=c\int_{\mathbb{R}}e^{i\lambda x+i\mu(-1)^{\frac{m-1}{2}}\lambda^{m}t}W(d\lambda), \, x\in\mathbb{R},\, t>0.\]
However, this formally defined field does not make sense as \[{\rm Cov}(\tilde{U}_{0,1}(t,x), \tilde{U}_{0,1}(t',x'))=c^{2}\int_{\mathbb{R}}\cos(\lambda(x-x')+\mu(-1)^{\frac{m-1}{2}}\lambda^{m}(t-t'))d\lambda,\] and, therefore, ${\rm Var}(\tilde{U}_{0,1}(t,x))=+\infty.$

Similar, formally following the proof of Theorem~{\rm{\ref{the32}}}, the limit field would be given as \[\tilde{U}_{0,2}(t,x)=c\int_{\mathbb{R}}\frac{e^{i\lambda x+i\mu (-1)^{\frac{m-1}{2}}\lambda^{m}t}}{|\lambda|^{\frac{1-\kappa_0}{2}}}dW(\lambda),\]
and its covariance function as \[{\rm Cov}(\tilde{U}_{0,2}(t,x), \tilde{U}_{0,2}(t',x'))=c^{2}\int_{\mathbb{R}}\frac{\cos({\lambda (x-x')+i\mu (-1)^{\frac{m-1}{2}}\lambda^{m}(t}-t'))}{|\lambda|^{{1-\kappa_0}}}d\lambda,\] which again implies that  ${\rm Var}(\tilde{U}_{0,2}(t,x))=+\infty.$

\end{Remark}

To investigate the limit behavior of solutions to odd-order heat equations, we introduce an averaging approach.  Note that kernel-smoothed solutions appear in various applications. First, in practice, no measurement devices can capture a value of a physical field at a precise time-space point. Instead, due to their finite resolution, they record averaged values. Second, kernel smoothing serves as a regularization technique to reduce irregularity of solutions, high-frequency oscillations, and noise. In general, kernel-smoothing can be considered as a specific case of filtering.

Let us consider the Gaussian-type kernel
$g( x_1,x) = e^{-(x_1 - x)^2}$ and define a spatially averaged version of the solution as
\begin{equation}\label{sqr}
   {\overline{U}}_{\varepsilon}^g(t,x) = \varepsilon^{-1/4} \left[ \int_{\mathbb{R}} g(x_1,x) u\left(\frac{t_1}{\sqrt{\varepsilon^m}}, \frac{x_1}{\sqrt{\varepsilon}} \right)  dx_1 - \mathbb{E}\int_{\mathbb{R}} g(x_{1},x)u\left(\frac{t_1}{\sqrt{\varepsilon^m}}, \frac{x_1}{\sqrt{\varepsilon}} \right) \right]dx_1.
\end{equation}
\begin{theorem}\label{thm5.2}
    Consider a random field $u(t,x), \, t>0,\,x\in\mathbb{R},$ which is the solution~\eqref{+} to the initial value problem \eqref{eq15} and \eqref{initial Airy}. Let $\eta(x),\, x\in \mathbb{R}$, has a covariance function satisfying Assumption~{\rm \ref{AssB}} with $A_{0}=0$. If $\varepsilon\to 0$, then the finite-dimensional distributions of the random field $\overline{U}_{\varepsilon}^{g}(t,x)$ weakly converge to the finite-dimensional distributions of the zero-mean Gaussian random field
    \begin{align*}
        {U}_{0}^g(t,x)=\frac{1}{2\sqrt{\pi}}\sqrt{\sum_{j=1}^{n}c_{1}(\kappa_{j})A_{j}K_{\frac{\kappa_{j}-1}{2}}(|w_{j}|)|w_{j}|^{\frac{\kappa_{j}-1}{2}}}\int_{\mathbb R}e^{i(\lambda x+\mu (-1)^{\frac{m-1}{2}}\lambda^{m}t)}e^{-\lambda^{2}/4}W(d\lambda)
    \end{align*}
    with the covariance function
    \begin{align*}
        {\rm Cov}(U_{0}^{g}(t,x),U_{0}^{g}(t',x'))&=\sum_{j=1}^{n}\frac{c_{1}(\kappa_{j})A_{j}K_{\frac{\kappa_{j}-1}{2}}(|w_{j}|)}{4\pi|w_{j}|^{\frac{1-\kappa_{j}}{2}}}\\&\times\int_\mathbb{R}\cos(\lambda(x-x')+\mu (-1)^{\frac{m-1}{2}}(t-t')\lambda^{m})e^{-\lambda^{2}/2}d\lambda.
    \end{align*}
\end{theorem}
\begin{proof}
    Notice, that
    \begin{align*}
   &\int_{\mathbb{R}} g(x_1,x) e^{i\lambda x_1 / \sqrt{\varepsilon}} \, dx_1
= \int_{\mathbb{R}} e^{-(x_1 - x)^2} e^{i\lambda (x_1-x) / \sqrt{\varepsilon}}e^{i\lambda x/\sqrt{\varepsilon}} \, dx_1 \\& = e^{i\lambda x / \sqrt{\varepsilon}} \int_{\mathbb{R}} e^{-u^2} e^{i\lambda u / \sqrt{\varepsilon}} \, du
= \sqrt{\pi} e^{i\lambda x / \sqrt{\varepsilon}} e^{-\lambda^2 / 4\varepsilon}.
\end{align*}
Thus, as $e^{-(x_{1}-x)^{2}}\sqrt{f(\lambda)}\in L_{2}(\mathbb{R}^2)$, by the stochastic Fubini theorem, the order of integration in \eqref{sqr} can be changed, which gives
\begin{align*}
    \overline{U}_{\varepsilon}^{g}(t,x) &=\frac{1}{2\pi\varepsilon^{1/4}}\int_{\mathbb{R}}g(x,x_1)\int_\mathbb{R}e^{i\lambda x_{1}/\sqrt{\varepsilon}+i\mu(-1)^{\frac{m-1}{2}}\lambda^{m}t/\sqrt{\varepsilon^m}}\sqrt{f(\lambda)}W(d\lambda)dx_1 \\&=\frac{\varepsilon^{-1/4}}{2\sqrt{\pi}} \int_{\mathbb{R}} e^{i\lambda x / \sqrt{\varepsilon}} e^{i\mu(-1)^{\frac{m-1}{2}} \lambda^m t / \sqrt{\varepsilon^m}} e^{-\lambda^2 / (4\varepsilon)} \sqrt{f(\lambda)} \, W(d\lambda).
\end{align*}

By using the change of variables $\lambda = \tilde{\lambda} \sqrt{\varepsilon}$, one obtains
\begin{align*}
    \overline{U}_{\varepsilon}^{g}(t,x)& = \frac{1}{2\sqrt{\pi}} \int_{\mathbb{R}} e^{i \tilde{\lambda} x }
e^{i \mu(-1)^{\frac{m-1}{2}} \tilde{\lambda}^m t} e^{- \tilde{\lambda}^2 / 4} \\&\times
\sqrt{\sum_{j=1}^{n} \frac{c_2(\kappa_j)}{2}A_{j} \left(
\frac{1 - \theta_{\kappa_j}(|\tilde{\lambda} \sqrt{\varepsilon} + w_j|)}{|\tilde{\lambda} \sqrt{\varepsilon} + w_j|^{1 - \kappa_j}}
+ \frac{1 - \theta_{\kappa_j}(|\tilde{\lambda} \sqrt{\varepsilon} - w_j|)}{|\tilde{\lambda} \sqrt{\varepsilon} - w_j|^{1 - \kappa_j}} \right) }
\, W(d\tilde{\lambda}).
\end{align*}
Then,
\[
U_{0}^{g}(t,x) = \frac{1}{2\sqrt{\pi}}  \sqrt{ \sum_{j=1}^{n}\frac{c_2(\kappa_j)}{w_{j}^{1 - \kappa_j}}A_j
(1 - \theta_{\kappa_j}(|w_j|)) }\int_{\mathbb{R}}e^{i {\lambda} x} e^{i \mu (-1)^{\frac{m-1}{2}}{\lambda}^m t} e^{-{\lambda}^2 / 4} W(d{\lambda}),
\]
and analogously to \eqref{131} one obtains that
\begin{align*}
    R(t,x) &= \mathbb{E} \left( \overline{U}_{\varepsilon}^{g}(t,x) - U_{0}^{g}(t,x) \right)^2
\\&= \frac{1}{4\pi}\int_{\mathbb{R}} e^{-\lambda^2/2}
\left( \sqrt{ Q_\varepsilon({\lambda}) } - \sqrt{\sum_{j=1}^{n}\frac{c_2(\kappa_{j})}{|w_{j}|^{1-\kappa_{i}}}(1-{\theta_{\kappa_j}}(|w_{j}|))A_{j}}\right)^2 d{\lambda}.
\end{align*}

Therefore, the remaining part of the proof is analogous to the proof of Theorem~\ref{the31} for the case of~$m=2$.
\end{proof}
\begin{example}\label{ex4} The values of $A_j,$ $\kappa_j,$ $w_j$ are the same as in Example~{\rm \ref{ex2}}. The parameter $\mu$ is chosen to be $1$. For the third-order heat equation, Figure~{\rm \ref{fig:a3}} presents a realization of the random field $U_{0}^{g}(t,x)$ obtained using an approximation of the stochastic integral with $\Delta = 0.001$ and $N = 4000$. Compared to Example~{\rm \ref{ex2}}, the field exhibits pronounced random oscillations over the entire region. As shown in Figure~{\rm \ref{fig:b3}}, the covariance function decays smoothly but oscillates with increasing spatial separation, in contrast to Example~{\rm \ref{ex2}}.

Figure~{\rm \ref{fig:c3}} shows that the temporal covariance for fixed $x-x' = 1.5$ is initially greater for smaller~$m$, but then decreases sharply, eventually becoming lower for small~$m$ as the time difference increases further. Compared to Example~{\rm \ref{ex2}}, Figure~{\rm \ref{fig:d3}} shows asymmetric spatial covariance for fixed $t-t' = 5$, which is monotonic on one side and oscillatory on the other, with a slight shift depending on $m$.
\begin{figure}[htbp]
\centering
\begin{subfigure}{0.5\textwidth}
    \centering
    \includegraphics[width=1.1\linewidth,trim=5.8cm 1.6cm 2.0cm 4.3cm,clip]{"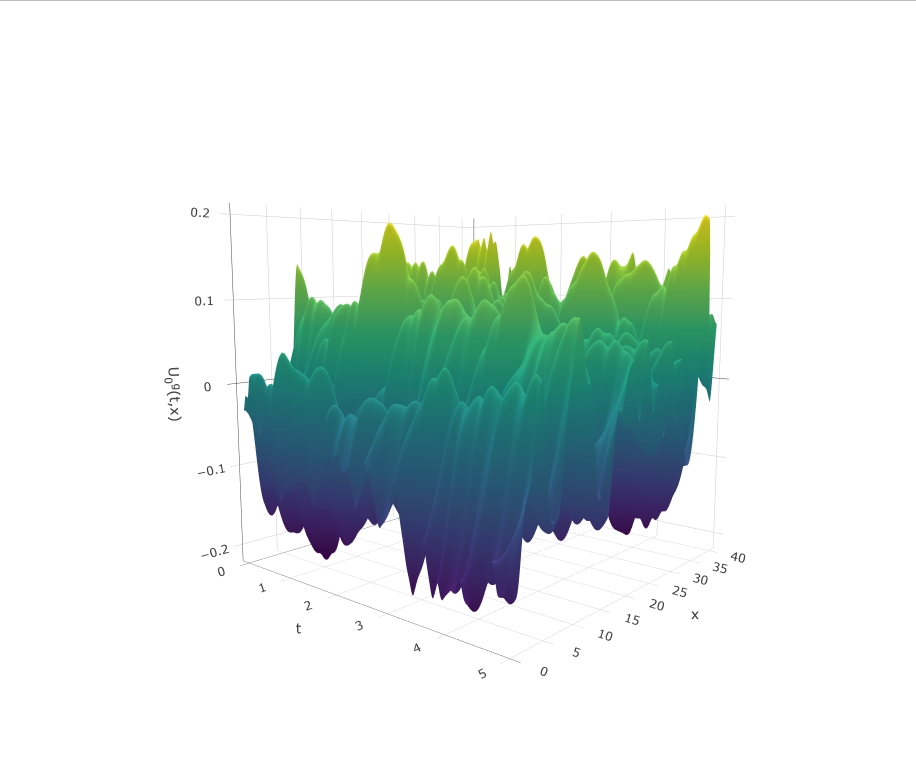"}
    \caption{Realization of $U_{0}^{g}(t,x)$}
    \label{fig:a3}
\end{subfigure}\hfill
\begin{subfigure}{0.5\textwidth}
    \centering
    \includegraphics[width=1.1\linewidth,trim=8.0cm 0.5cm 5cm 3.0cm,clip]{"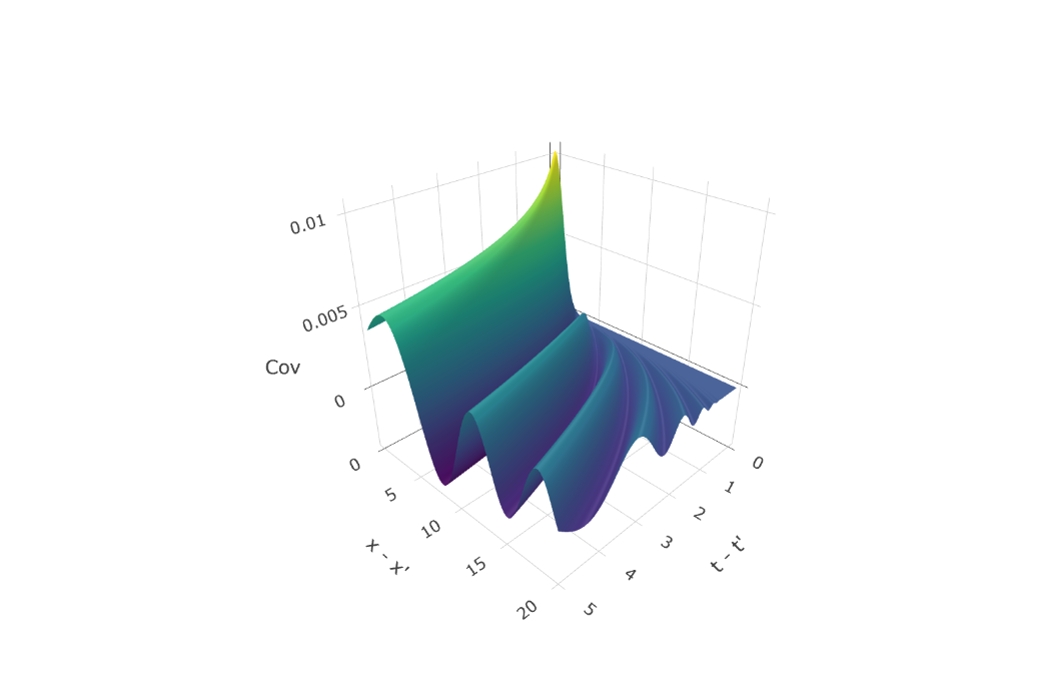"}
    \caption{Covariance function of $U_{0}^{g}(t,x)$}
    \label{fig:b3}
\end{subfigure}
\caption{The case of the 3rd-order heat equation with $A_{0}=0$}
\label{fig6}
\end{figure}

\begin{figure}[htbp]
\centering
\begin{subfigure}{0.485\textwidth}
    \centering
    \includegraphics[width=\linewidth]{"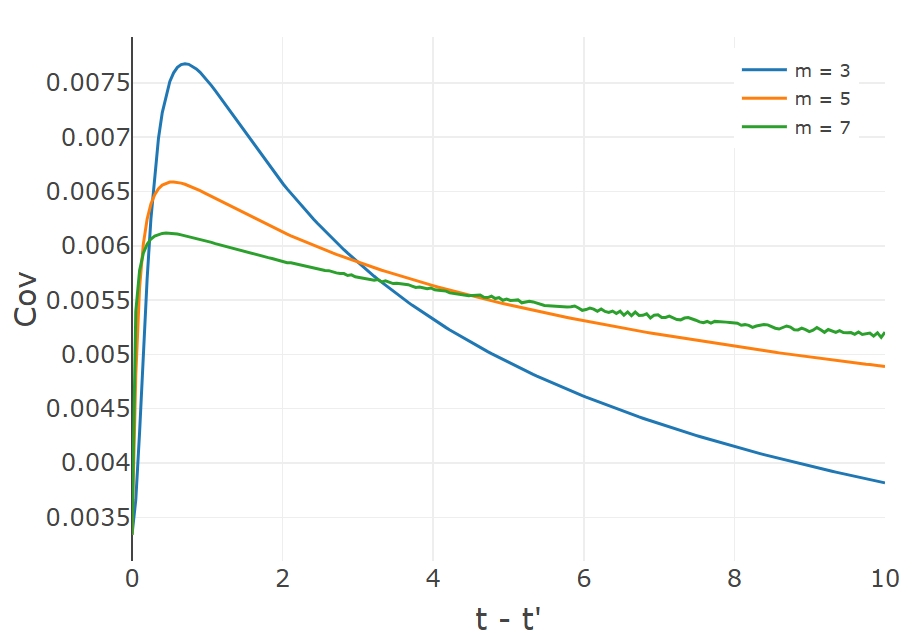"}
    \caption{Temporal covariance function,\,$x-x'=1.5$}
    \label{fig:c3}
\end{subfigure}\hfill
\begin{subfigure}{0.48\textwidth}
    \centering
    \includegraphics[width=\linewidth]{"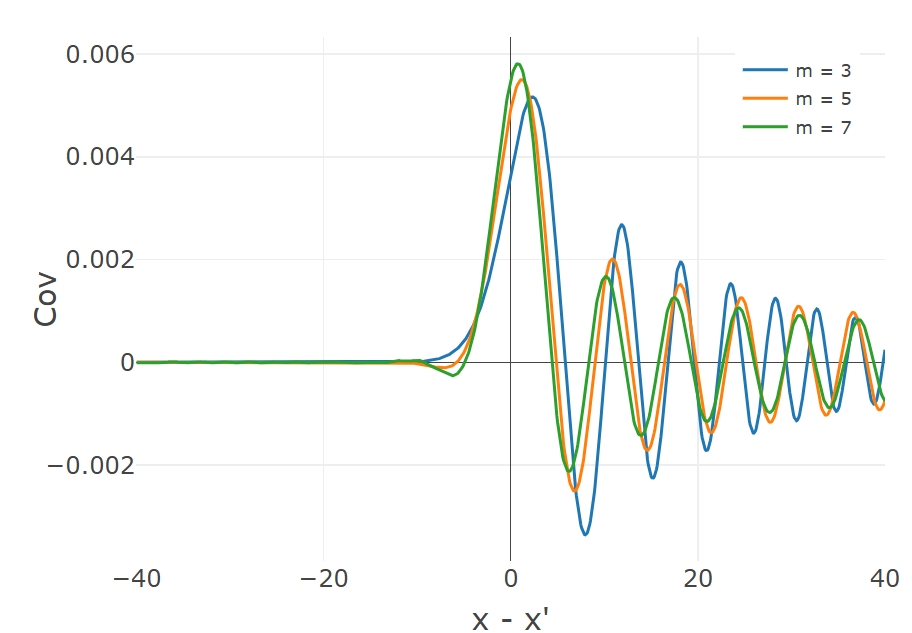"}
    \caption{Spatial covariance function, $t-t'=5$}
    \label{fig:d3}
\end{subfigure}
\caption{Covariance structures of $U_{0}^{g}(t,x)$ for $m$th-order heat equations with $A_{0}=0$}
\label{fig7}
\end{figure}
\end{example}
\begin{theorem}\label{thm 5 odd}
    Let a random field $u(t,x), \, t>0, \, x\in \mathbb{R}$,  be the solution to the initial value problem \eqref{eq15} and \eqref{initial Airy}. Let $\eta(x),\, x\in\mathbb{R}$, have covariance function that satisfies Assumption~{\rm \ref{AssB}} with $A_{0}\neq 0$. If $\varepsilon\to 0$, then the finite-dimensional distributions of the random field $\overline{U}_{\varepsilon}^{g}(t,x)$ weakly converge to the finite-dimensional distributions of the zero-mean Gaussian random field
    \[U_{0}^{g}(t,x)=\frac{\sqrt{A_{0}\,c_2(\kappa_{0})}}{2\sqrt{\pi}}\int_{\mathbb{R}}\frac{e^{i(\lambda x+\mu (-1)^{\frac{m-1}{2}}\lambda^{m}t)}}{|\lambda|^{\frac{1-\kappa_{0}}{2}}}e^{-\lambda^{2}/4}W(d\lambda)\]
    with the covariance function
    \[{\rm Cov}(U_{0}^{g}(t,x),U_{0}^{g}(t',x'))=\frac{A_{0}\,c_{2}(\kappa_{0})}{4\pi}\int_{\mathbb{R}}\frac{\cos(\lambda(x-x')+\mu(-1)^{\frac{m-1}{2}}(t-t')\lambda^{m})}{|\lambda|^{1-\kappa_{0}}}e^{-\lambda^{2}/2}d\lambda.\]
\end{theorem}
\begin{proof}
    The result follows from steps analogous to those in the proofs of Theorems~\ref{the32} and \ref{thm5.2}.
\end{proof}
\begin{Remark}
    The random fields $U_{0}^{g}(t,x)$ in Theorem {\rm\ref{thm5.2}} and {\rm\ref{thm 5 odd}} are stationary both in space and time.
\end{Remark}
\begin{Remark}
Theorems {\rm\ref{thm5.2}} and {\rm\ref{thm 5 odd}} extend and correct results in {\rm\cite{beghin2000}}.
\end{Remark}

\begin{example}
This example uses the same parameter values as Example~{\rm\ref{ex3}} and the equation parameter $\mu=1$.
The simulated realization of $U_{0}^{g}(t,x)$ in Figure~{\rm \ref{fig:a4}} exhibits an oscillating pattern as in Example~{\rm \ref{ex4}}, but with values increasing in the spatial direction in this case.
The covariance function in Figure~{\rm\ref{fig:b4}} exhibits similar waves to those in Example~{\rm\ref{ex4}}, but takes on larger values.

The spatial covariance in Figure~{\rm\ref{fig:d4}} exhibits oscillations around a monotonically decreasing average for increasing positive differences $x-x'$, and monotonically decreasing covariances for the increasing magnitude of negative differences. The function does not vanish as rapidly as the one in Example~{\rm\ref{ex4}}.
\begin{figure}[htbp]
\centering
\begin{subfigure}{0.5\textwidth}
    \centering
    \includegraphics[width=1.1\linewidth,trim=5cm 2cm 2cm 6cm,clip]{"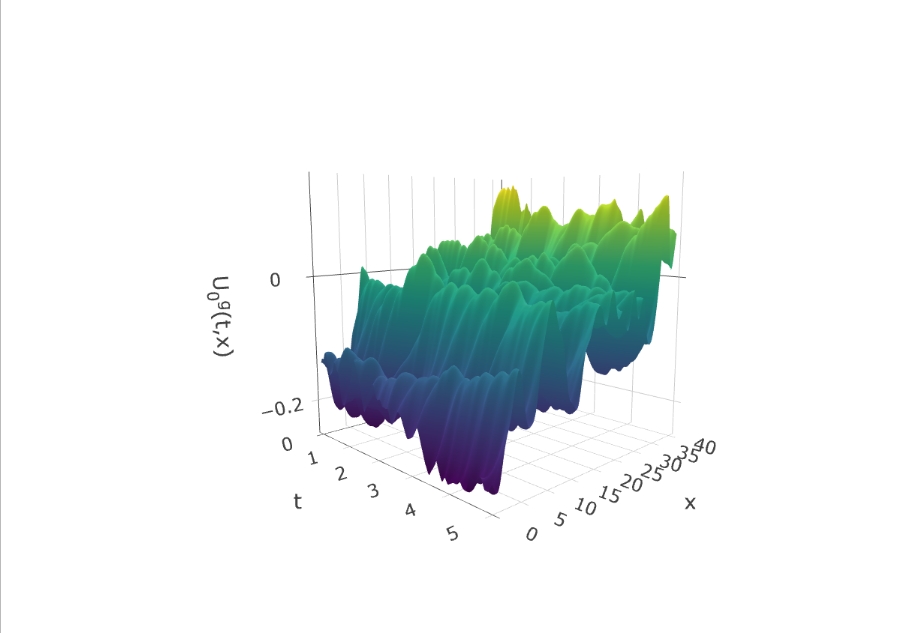"}
    \caption{Realization of $U_{0}^{g}(t,x)$}
    \label{fig:a4}
\end{subfigure}\hfill
\begin{subfigure}{0.5\textwidth}
    \centering
\includegraphics[width=1.1\linewidth,trim=3.5cm 2cm 1cm 6cm,clip]{"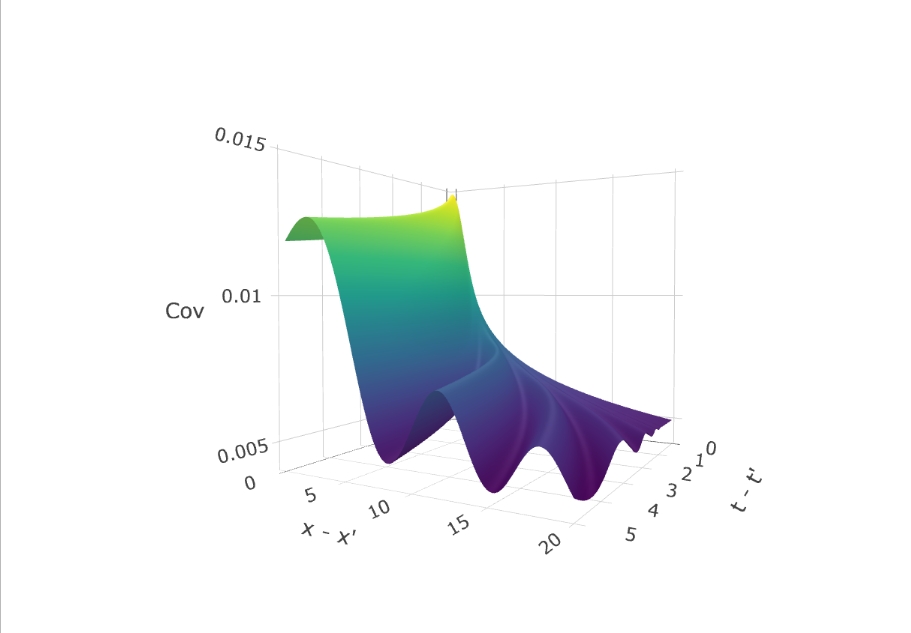"}
    \caption{Covariance function of $U_{0}^{g}(t,x)$}
    \label{fig:b4}
\end{subfigure}
\caption{The case of the 3rd-order heat equation with $A_{0}\neq 0$}
\end{figure}

\begin{figure}[htbp]
\centering
\begin{subfigure}{0.485\textwidth}
    \centering
    \includegraphics[width=\linewidth]{"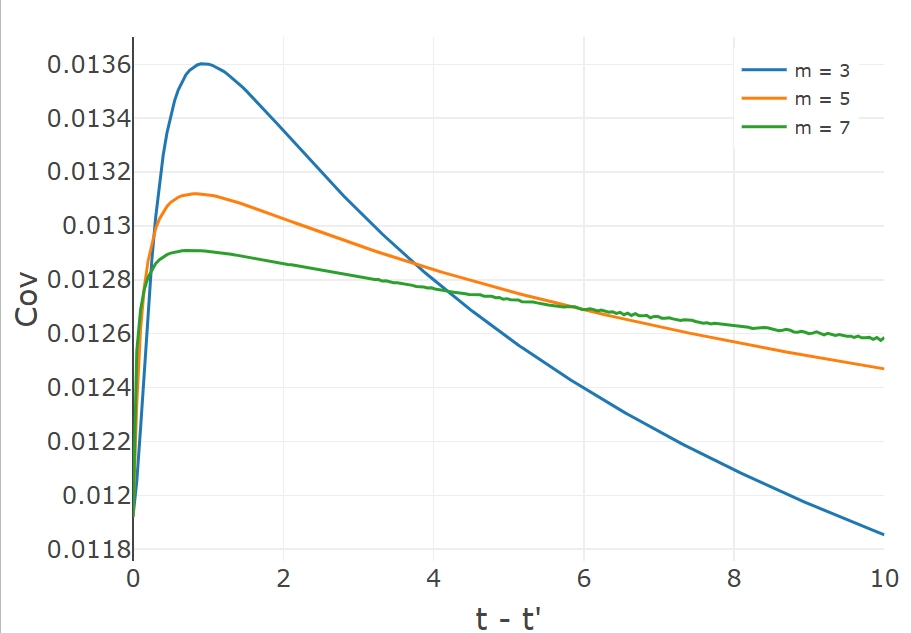"}
    \caption{Temporal covariance function,\,$x-x'=1.5$}
    \label{fig:c4}
\end{subfigure}\hfill
\begin{subfigure}{0.48\textwidth}
    \centering
    \includegraphics[width=\linewidth]{"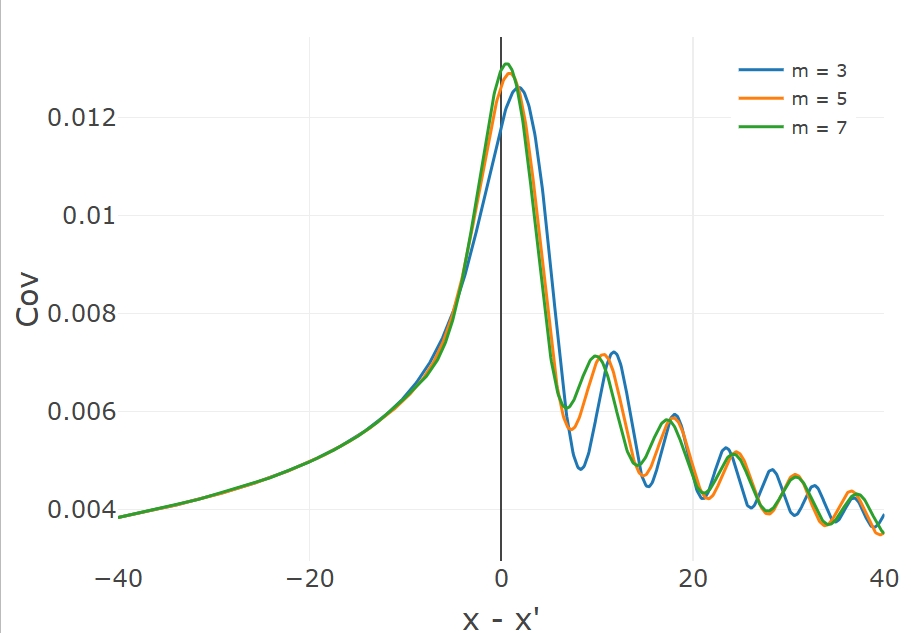"}
    \caption{Spatial covariance function, $t-t'=5$}
    \label{fig:d4}
\end{subfigure}
\caption{Covariance structures of $U_{0}^{g}(t,x)$ for $m$th-order heat equations with $A_{0}\neq 0$}
\end{figure}

\end{example}
\section{Conclusion}\label{sec conclusion}
 \noindent The paper proves multiscaling limit theorems for solutions to high-order PDEs with random initial conditions exhibiting long-memory and cyclic behavior. It analyses Airy-type and heat equations of various orders for cases of initial conditions possessing spectral singularities at zero and non-zero frequencies. Using spectral and scaling methods, it is proved that their rescaled solutions converge to Gaussian random fields. The limit fields depend on the orders of the equations and the presence or absence of the spectral singularity at zero.
Some of the open problems motivated by the obtained results are
\renewcommand{\labelitemi}{-}
\begin{itemize}
    \item Derive limit theorems analogous to Theorem \ref{thm5.2} for general weighted functionals of the solutions, see \cite{Statisticalanalysis, Ivanov2013, Olenko2013};
    \item Generalize the results for $ \eta(x)$ of an arbitrary Hermite rank, see \cite{beghin2000};
    \item Extend the results to the class of random initial conditions and limit theorems with slowly varying functions, see \cite{ANH2000239, Anh2002RenormalizationAH, LeonenkoOlenko2013};\item Obtain similar results for general fractional Riesz-Bessel equations, see \cite{ANH2000239, Anh2002RenormalizationAH, alghamdi2024}; \item Generlize the approach to the multidimensional case of random fields, see \cite{ANH2000239, Anh2002RenormalizationAH, Olenko2013}.
\end{itemize} \section*{Acknowledgments}
This research was supported by the Australian Research Council's Discovery Projects funding scheme (project number DP220101680).  A.~Olenko was partially supported by La Trobe University's SCEMS CaRE and Beyond grant. N.~Leonenko would like to thank for support and hospitality during the programmes “Fractional Differential Equations” (FDE2), “Uncertainly Quantification and Modelling of Materials” (USM), both supported by EPSRC grant EP/R014604/1, and the programme “Stochastic systems for anomalous diffusion” (SSD), supported by EPSRC grant EP/Z000580/1, at Isaac Newton Institute for Mathematical Sciences, Cambridge. He was also partially supported under Croatian Scientific Foundation grant “Scaling in Stochastic Models” HRZZ-IP-2022-10-8081, grant FAPESP 22/09201-8 (Brazil) and the Taith Research Mobility grant (Wales, Cardiff University). A.~Olenko is grateful to Prof.~E.Orsingher (Sapienza Universit\`a di Roma) for discussions on the general theory of Airy processes.

\section*{Declarations}

{\bf Conflicts of interest:} The authors have no conflicts of interest.

\noindent{\bf  Data availability:} No datasets were generated or analysed during the current study.

\noindent{\bf Code availability:} All numerical computations, simulations and plotting in this paper were performed using the software R (version 4.5.1) and Maple~2023. The corresponding R
and Maple code is freely available in the folder
”Research materials” from the website \url{https://sites.google.com/site/olenkoandriy/}

\bibliography{Bibliography.bib}

\end{document}